\documentclass[11pt,a4paper]{article}
\usepackage{authblk}
\usepackage{graphicx}
\usepackage{amssymb}
\usepackage{amsmath}
\usepackage{amsthm}
\usepackage{booktabs}
\usepackage{paralist}
\usepackage{bm}
\usepackage{url}
\usepackage{color}
\usepackage{mathtools}
\usepackage{enumitem}
\usepackage{algorithm2e}
\usepackage{pgfplots}
\usepackage{array}
\usepackage{multirow}

\usepackage{lmodern}

\usepackage{todonotes}
\mathtoolsset{showonlyrefs}
\usepackage[justification=centering]{caption}
\usepackage{subcaption}
\usetikzlibrary{patterns}
\usepackage{tikz-network}

\usetikzlibrary{decorations.pathreplacing}

 \usepackage[breaklinks,bookmarks = false]{hyperref}
 \hypersetup{colorlinks, linkcolor = blue, citecolor = blue,
 	urlcolor = blue, pagecolor = blue, plainpages = false, pdfwindowui = false}

\newtheorem{lemma}{Lemma}[section]
\newtheorem{corollary}[lemma]{Corollary}
\newtheorem{proposition}[lemma]{Proposition}
\newtheorem{assumption}[lemma]{Assumption}
\newtheorem{theorem}[lemma]{Theorem}
\newtheorem{definition}[lemma]{Definition}
\newtheorem{remark}[lemma]{Remark}

\newcommand\Om{\Omega}
\newcommand\GD{{\Gamma_{\mathrm D}}} 
\newcommand\GN{{\Gamma_{\mathrm N}}} 

\newcommand\Ho{H^1(\Om)}
\newcommand\Hoo{H^1_0(\Om)}
\newcommand\HoD{H^1_{0,\mathrm{D}}(\Om)}

\newcommand\Hmo{H^{-1}(\Om)}

\newcommand\Th{\mathcal{T}_h}

\newcommand\tx{{\bm{x}}}

\newcommand\tb{\bm{b}}

\newcommand\tr{\bm{r}}

\newcommand\tv{\bm{v}}

\newcommand\tA{\bm{A}}

\newcommand\tG{\bm{G}}

\newcommand\tW{\bm{W}}

\newcommand{\bal}{{\bm{\alpha}}}

\newcommand\I{\mathcal{I}}
\newcommand\J{\mathcal{J}}

\newcommand\LL{\mathcal{L}}

\newcommand\N{\mathcal{N}}

\newcommand\R{\mathcal{R}}

\newcommand\V{\mathcal{V}}

\newcommand\RR{{\mathbb R}}
\newcommand\NN{{\mathbb N}}

\newcommand\dx{\, \mathrm{d} \tx}

\newcommand\dxx{\, \mathrm{d} x}

\definecolor{mycolor1}{RGB}{168,220,232}
\definecolor{mycolor2}{RGB}{168,207,184}
\definecolor{mycolor3}{RGB}{168,168,245}
\definecolor{mycolor4}{RGB}{220,168,232}

\title{Adaptive Deep Fourier Residual method via overlapping domain decomposition}

\author[1]{Jamie M. Taylor} 
\author[2]{Manuela Bastidas}
\author[3]{Victor M. Calo}
\author[2,4,5]{David Pardo}

\affil[1]{CUNEF Universidad, Madrid, Spain. \protect\\ \texttt{jamie.taylor@cunef.edu}}
\affil[2]{University of the Basque Country (UPV/EHU), Leioa, Spain.}
\affil[3]{Curtin University, Perth, WA, Australia.}
\affil[4]{Basque Center for Applied Mathematics (BCAM), Bilbao, Spain.}
\affil[5]{Ikerbasque (Basque Foundation For Sciences), Bilbao, Spain.}

\graphicspath{{Figures/}}

\begin{document}
 \maketitle

 
 \begin{abstract}
	The Deep Fourier Residual (DFR) method is a specific type of variational physics-informed neural networks (VPINNs). It provides a robust neural network-based solution to partial differential equations (PDEs). The DFR strategy is based on approximating the dual norm of the weak residual of a PDE. This is equivalent to minimizing the energy norm of the error. To compute the dual of the weak residual norm, the DFR method employs an orthonormal spectral basis of the test space, which is known for rectangles or cuboids for multiple function spaces.
	
	In this work, we extend the DFR method with ideas of traditional domain decomposition (DD). This enables two improvements: (a) to solve problems in more general polygonal domains, and (b) to develop an adaptive refinement technique in the test space using a Döfler marking algorithm. In the former case, we show that under non-restrictive assumptions we retain the desirable equivalence between the employed loss function and the $H^1$-error, numerically demonstrating adherence to explicit bounds in the case of the L-shaped domain problem. In the latter, we show how refinement strategies lead to potentially significant improvements against a reference, classical DFR implementation with a test function space of significantly lower dimensionality, allowing us to better approximate singular solutions at a more reasonable computational cost.
	
 \end{abstract}

 \section{Introduction}
 
The design of an accurate scheme for solving partial differential equations (PDEs) with machine learning techniques, explicitly using artificial neural network models (NNs)~\cite{ lagaris1998artificial}, requires a proper NN architecture, optimization method, and loss function to minimize. When it comes to the latter aspect, determining a suitable loss function is often complicated, and different machine learning models used for solving PDEs employ distinct loss functions. For instance, in methods like PINNs (Physics-Informed Neural Networks)~\cite{ jagtap2020conservative, karniadakis2021physics, raissi2019physics, sirignano2018dgm}, the definition of the loss function relies on the strong formulation of the PDE. In contrast, in other methods like VPINNs, (Variational Physics-Informed Neural Networks)~\cite{ kharazmi2019vpinns, kharazmi2021hp, khodayi2020varnet, khodayi2020deep}, the loss function uses the weak formulation. All these strategies have their challenges; for example, the methods based on the strong formulation of the PDE, such as PINNs, fail when dealing with PDEs that inherently yield low-regularity solutions, and in the classical VPINNs, the PDE residual norms employed as a loss function are not consistent with the dual norm of the error; thus, their accuracy highly depends on the choice of the set of functions in the test space~\cite{ badia2024finite, berrone2022solving, berrone2022variational}. 
As a remedy, the authors in~\cite{ rojas2023robust} detail a theoretical framework in which an appropriate loss function for a VPINN-type model is a well-defined approximation of the dual norm of the PDE's weak residual. Specifically, many classical PDEs can be reformulated in a weak form, casting them in terms of the weak residual $\R: H \to H^*$, where $H$ is a Hilbert space and $H^*$ is the corresponding dual space. In fact, for many well-posed problems, like the classical Poisson and time-harmonic Maxwell's equations, the dual norm of the weak residual is a good estimate of the error's norm. In this sense, if $u^*$ is the exact solution of the PDE, then there exist constants $0 < \gamma < M$ such that 
\begin{equation}\label{eqUpLowerBd}
	\frac{1}{M} \| \R(u) \|_{H^*} \leq \|u- u^*\|_H \leq \frac{1}{\gamma} \| \R(u) \|_{H^*},
\end{equation}
where $u^*$ is the exact solution and $u\in H$ is arbitrary. In~\cite{ rojas2023robust}, it is shown that the choice of the loss function as an approximation of $\| \R(u) \|_{H^*}$ in a finite-dimensional subspace of $H^*$ leads to a Robust version of Variational Physics-Informed Neural Networks (RVPINNs), in which the discrete loss is correlated with the $H$-norm of the error up to an oscillation term.

The Deep Fourier Residual (DFR) method proposed in~\cite{ Manuela2023deep, taylor2023deep} is a particular case of  RVPINNs, in which the approximation of the dual norm of the weak residual $\| \R(u) \|_{H^*}$ is calculated via a spectral representation of the residual. The advantage of the DFR method is that it avoids the primary bottleneck of RVPINNs by eliminating the need to compute (and invert) the Gram matrix corresponding to the chosen basis of the finite-dimensional test space. The loss computation heavily relies on the knowledge of an orthonormal basis for the space $H$. In~\cite{ Manuela2023deep, taylor2023deep}, the basis functions are given as eigenvectors of the Laplacian, which are generally unknown/not available, with certain exceptions in simple geometries such as $n$-rectangular domains (i.e., rectangles in 2D or rectangular cuboids in 3D). Indeed, the construction of such a basis is non-trivial for general domains. 

Here, we present an approach that can be used beyond the setting of $n$-rectangular domains. Based on the ideas of traditional domain decomposition (DD) methods~\cite{ chan1994domain, dolean2015introduction, toselli2004domain}, we represent the PDE's domain $\Om$ by a finite union of overlapping $n$-rectangles. Through this construction, we define an auxiliary norm, the $\star$-norm, that approximates the dual norm of the weak residual and yields estimates (i.e.,~\eqref{eqUpLowerBd}) based on reasonable assumptions. We apply the DD over the \textit{test} space rather than over the \textit{trial} space. As a result, we do not iterate between subdomains or use transmission conditions to solve subproblems, unlike other PINN-like techniques that have used the ideas of DD methods~\cite{ heinlein2021combining, moseley2023finite, shukla2021parallel}. 

By partitioning the subdomain $\Om$ as an overlapping union of $n$-rectangles, we can also address PDEs with solutions with singularities or discontinuities due to factors like non-smooth parameters or localized sources. In those cases, the classic DFR method fails to provide accurate solutions because it requires many Fourier modes with support over the entire domain to represent phenomena that may occur only locally. However, by developing a subdomain decomposition of $\Om$, we overcome this problem and focus on the regions of the domain where errors are both local and high frequency (singularities, interfaces, etc.). We propose using an error indicator to drive local subdomain refinements that captures such errors. Our error indicator is a local contribution of the loss function making the training and refinement consistent, improving the structure of the initial subdomain decomposition of $\Om$, and reducing errors in regions with potential singularities, high gradients in the solution or high-frequency errors. Similar strategies for developing VPINNs combined with domain decomposition were proposed in~\cite{kharazmi2021hp}. Our contribution aligns with~\cite{kharazmi2021hp} in that adaptivity is performed over the test space. However, unlike in~\cite{kharazmi2021hp}, herein we consider spectral test functions; the residual is measured in a new $\star$-norm that we introduce in Section~\ref{sec:starnorm}, and we have overlapping subdomains. Additionally, in our implementations, we accelerate the training process by using a hybrid optimization method, where the weights in the last layer of the neural networks are found in one iteration by solving a least-squares (LS) system of equations~\cite{ cyr2020robust, huang2006extreme}. 

Integration errors are commonly encountered in PINN-related methodologies and usually lead to slow convergence of the optimization algorithm and inaccurate or unreliable values of the loss function~\cite{ weinan2021algorithms, rivera2022quadrature, wu2023comprehensive}. In the numerical examples with singularities, addressing the numerical integration problems becomes a sensitive matter. Thus, we employ overkill mesh-integration techniques, and we leave the design of more sophisticated quadrature rules as future work (see~\cite{ magueresse2023adaptive, wu2023comprehensive} as possible avenues to improve quadrature).

We structure this work as follows: Section~\ref{sec:preliminaries} describes the notation and the definition of the weak-form residual for problems with solutions in the space $\Ho$. Sections~\ref{sec:DFR0},~\ref{sec:discLoss} and~\ref{sec:optimis} detail the structure of the proposed NN models, the framework of the DFR method on rectangular domains and the optimization method. Section~\ref{sec:DFRDDteo} presents the theoretical basis for using domain decomposition ideas at the core of the DFR method. In Section~\ref{sec:extension}, we extend the DFR method to polygonal domains and in Section~\ref{sec:localRef}, we propose using local refinements to improve the method's accuracy on specific regions of $\Om$. In Section~\ref{sec:num}, we demonstrate the capabilities of the DFR method in 1D and 2D, solving problems with low-regularity solutions. Section~\ref{sec:conc} concludes and discusses future work. In the appendices~\ref{app:2},~\ref{app:1} and~\ref{app:3}, we outline some of the technical results related to norms we employ in our work.

\section{Variational formulation}\label{sec:preliminaries}

For general domains $\Om \subset \RR^n$ with $n =1, 2,$ or $3$ and boundary $\partial \Om$, we denote by $L^2(\Om)$ the space of squared-integrable real-valued functions equipped with the usual norm. For a general Hilbert space $H$, $(\cdot, \cdot)_H$ represents the inner product, the corresponding norm is $ \| v \| _H^2 := (v, v)_H$, and $\langle \cdot, \cdot \rangle_{H^*\times H}$ denotes the duality pairing between elements in $H$ and its dual space $H^*$.

The space $\Ho$ is the space of $L^2(\Om)$ functions having weak derivatives in $L^2(\Om)$. If the boundary of $\Om$ is Lipschitz and $\GD,\GN$ are disjoint subsets of $\partial\Om$ such that $\partial\Om=\overline{\GD}\cup\overline{\GN}$, there exist trace operators $\gamma_0(u) = u|_{\partial \Om}$ and $\gamma_\mathrm{D}(u) = u|_{\GD}$ of $u\in \Ho$ that allow these space definitions
\begin{equation}
		\Hoo :=  \left\lbrace u \in \Ho: \, \gamma_0(u) = 0 \right\rbrace, \quad \text{and} \quad
		\HoD :=  \left\lbrace u \in \Ho: \, \gamma_\mathrm{D}(u) = 0 \right\rbrace.
\end{equation}
We denote $\Hmo$ the dual space of $\Hoo$ and $H^{*}_{\mathrm{D}}(\Om)$ the dual space of $\HoD$. 
 


For the sake of simplicity, we concentrate our attention on general PDEs with solutions in the space $\HoD$. An example of such PDE is Poisson's equation that can be written in weak form as: find $u\in \HoD$ satisfying
\begin{equation}\label{eq:probpois}
	\int_\Om \nabla u\cdot \nabla v+f \cdot v\dx -\int_{\GN}g_u \cdot v \dx = 0 \qquad \forall v\in \HoD,
\end{equation}
where $f\in L^2(\Om)$ and $g_u$ is the Neumann boundary data. 

In general and in terms of operators, solving the weak formulation of a PDE with solution in the space $\HoD$ reduces to finding $u \in  \HoD$ such that
\begin{equation}\label{eq:res0}
	\langle\R(u),v\rangle_{H^{*}_\mathrm{D}\times H^1_{0,\mathrm{D}}} = b(u,v)-\ell(v) = 0 \qquad \forall v \in \HoD,
\end{equation} 
where $\R:\HoD \to H^{*}_\mathrm{D}$ is the weak residual operator, $b: \HoD\times \HoD \to \RR$ is a bilinear form and $\ell \in H^{*}_\mathrm{D}$.  Moreover, if $u^*$ is the exact solution of Eq.~\eqref{eq:res0} and the bilinear form $b$ satisfies the continuity condition
\begin{equation}\label{eqContLM1}
	|b(u,v)| \leq M \| u \| _{\HoD} \| v \| _{\HoD}  \qquad \forall u,v \in \HoD,
\end{equation}
and the inf-sup stability condition 
\begin{equation}
		\inf\limits_{u\in \HoD\setminus \{0\}} \, \sup\limits_{v\in \HoD\setminus \{0\}}\frac{|b(u,v)|}{ \| u \| _{\HoD} \| v \| _{\HoD}}\geq \gamma,
\end{equation}
with $0<\gamma<M$, one may estimate the error using the dual norm of the residual via Eq.~\eqref{eqUpLowerBd} (for detailed proofs of these bounds see, e.g.,~\cite{ ciarlet2002finite, ciarlet2013linear}).

Now, from Eq.~\eqref{eqUpLowerBd}, we obtain that the dual norm of the residual is equivalent to the norm of the error. Therefore, if one aims to minimize the error, the dual norm of the residual is an attractive choice of minimization target (loss function).

\section{Discretization over the trial space using Neural Networks}\label{sec:DFR0}

A fully-connected feed-forward neural network is a mathematical model that consists of a collection of interconnected nodes. Each node is defined by a non-linear map applied to an affine combination of inputs $\tx \in \RR^n$ to certain outputs. The nodes are arranged in ($L$) layers, and the output of each layer can be described as
\begin{equation}\label{eq:layerspar}
	\N_j = \sigma_j(\tW_j\N_{j-1} + \tb_j ) \quad \text{ for } j \in \{ 1 ,\dots, L \},
\end{equation}
where $\sigma_j$ represents non-linear activation functions for $j \in \{1,\dots, L\text{-}1\}$, and $\sigma_L$ is the identity function. We take $\N_{0} = \tx$ and  $\tb_L=\mathbf{0}$ and denote our candidate solution $\tilde{u} := \N_L$. The parameters of the neural network are the weights $\tW_j \in \RR^{d_j \times d_{j-1}}$ and the biases $\tb_j \in \RR^{d_j}$, with $d_j$ being the number of nodes on each layer. The weights and biases are determined during a training process, which involves the application of a gradient-based optimization algorithm to minimize a specific discretized loss function, denoted as $\LL(u)$. Both the loss function and the optimization method are explained below.

When imposing homogeneous Dirichlet boundary conditions on our candidate solutions, we use a non-trainable cut-off function $\chi: \overline{\Om} \to \RR$. If $\tilde{u}$ is the output of the fully connected feed-forward neural network, then our final output is $u=\chi\tilde{u}$ where $\chi$ is a function satisfying $\chi|_{\Gamma_D}=0$ and $\chi>0$ on $\bar{\Om}\setminus\Gamma_D$.

\section{Discretization over the test space}\label{sec:discLoss}

Given that the dual norm of the residual is an estimator of the energy norm of the error, as Eq.~\eqref{eqUpLowerBd} shows, our primary goal is to choose $\LL(u)$ such that it accurately approximates $ \| {\R(u)} \|_{H^{*}_\mathrm{D}}$. First, we define $\V$, a finite-dimensional subspace of $\HoD$, as the span of some basis functions $\{\Phi_k\}_{k\in\J}$, with $\J$ being a finite set of (sorted) indices. 

From the Riesz representation theorem~\cite{ oden2017applied}, we obtain that for each $\R(u) \in H^{*}_\mathrm{D}$ there exists a corresponding $u_{\R} \in \HoD$, known as the \emph{Riesz representative} of $\R(u)$, such that 
\[ (u_{\R} ,v)_{\HoD}= \langle \R(u), v \rangle_{H^{*}_\mathrm{D}\times H^1_{0,\mathrm{D}}}, \, \forall v\in \HoD,\]
and \[ \| {\R(u)} \| _{H^{*}_\mathrm{D}(\Om)} = \| u_{\R} \| _{\HoD}.\]

Next, we project the Riesz representative of $\R(u)$ onto $\V$ and define $\hat{u}_{\R} := \sum\limits_{k\in\J} \alpha_k(u) \Phi_k$, which is the unique element in $\V$ satisfying
\begin{equation}\label{eq:rieszrep}
	(\hat{u}_{\R},\Phi_k)_{\HoD} = 	\langle\R(u),\Phi_k \rangle_{H^{*}_\mathrm{D}\times H^1_{0,\mathrm{D}}}  \quad \forall k \in \J.
\end{equation} 

This equation can be expressed in vectorial form as follows:
\begin{equation}\label{eq:vects}
	\tG \bal(u) = \tr(u),
\end{equation} 
where $\tG := [(\Phi_i, \Phi_j)_{\HoD}]_{i,j \in \J}$ is the Gram matrix, $\bal(u) := [\alpha_j(u)]_{j \in \J}$ is the vector of unknown coefficients, and $\tr(u) := [\langle\R(u),\Phi_j \rangle_{H^{*}_\mathrm{D}\times H^1_{0,\mathrm{D}}}]_{j \in \J}$. Consequently, we have
\begin{equation}\label{eq:gram}
	 \| {\R(u)} \| _{H^{*}_\mathrm{D}(\Om)}^2 \approx \| \hat{u}_{\R} \|_{\HoD}^2 = \bal(u)^T \tG \bal(u) = 
	\tr(u)^T \tG^{-1} \tr(u)
\end{equation} 
where the quality of the approximation depends on the capacity of the space $\V$ to approximate $\hat{u}_{\R}$.

Therefore, an appropriate choice of a discrete loss function is
\begin{equation}\label{eq:losseq0}
	\LL(u) := 	\tr(u)^T \tG^{-1} \tr(u). 
\end{equation}

We refer to~\cite{ rojas2023robust} for more details on the construction and the analysis of the loss function~\eqref{eq:losseq0}. The computation of Eq.~\eqref{eq:losseq0} requires the (off-line) construction and inversion of the Gram matrix, which can be computationally expensive and numerically unstable for certain choices of the subspace $\V$ and its basis functions. We address this computational challenge by concentrating on the ideas of the DFR method applied to rectangular domains, in which the basis is orthonormal, and the inversion of the Gram matrix is trivial. The construction of such a basis is described in~\cite{ taylor2023deep} in the case of solving problems in the Hilbert space $\HoD$ and in~\cite{ Manuela2023deep} for the case of solving the time-harmonic Maxwell's equations. Section~\ref{sec:dfrrect} summarizes this construction for $\HoD$ on rectangular domains, while Section~\ref{sec:extension} extends it to the case of polygonal domains.

\subsection{The DFR method on rectangular domains}\label{sec:dfrrect}

Let $\Om$ be a rectangular subset of $\RR^n$ (referred to as an $n$-rectangular domain), which can be expressed as the Cartesian product of $n$ intervals $\Om := \prod\limits_{j=1}^n(a_j,b_j)$ (or its rotations). In this case, a well-known set of basis functions for the space $\HoD$ is built using the eigenvectors of $(1-\Delta)$ (a classical construction described in~\cite{ davies1995spectral, taylor2023deep}). Specifically, the orthonormal basis functions of $\HoD$ in $n$-dimensions are the tensor products of the following functions in 1D:
\begin{equation}\label{eq:basis}
		\Phi_k = \frac{\varphi_k}{\| \varphi_k \|_{\HoD}},  \text{ with } \, {\varphi}_k(x) = \sqrt{\frac{2}{b-a}} \left\{\begin{array}{l l}
			\sin\left(k\left(\frac{\pi(x-a)}{b-a}\right)\right) & \GD = \{a,b\},\\
			\sin\left(\left(k-\frac{1}{2}\right)\left(\frac{\pi(x-a)}{b-a}\right)\right) & \GD = \{a\},\\
			\cos\left(\left(k-\frac{1}{2}\right)\left(\frac{\pi(x-a)}{b-a}\right)\right) & \GD = \{b\},\\
			\cos\left((k-1)\left(\frac{\pi(x-a)}{b-a}\right)\right) & \GD = \emptyset,
		\end{array}\right. 
	\end{equation}
where $k\in \I$, with $\I$ being an infinite index set.

Given a cut-off frequency, we choose a finite index subset  $\J$ and define  $\V = \text{span}\left( \{\Phi_k\}_{k \in \J}\right)$. Leveraging the orthogonality of the functions in the set $\{\Phi_k\}_{k \in \J}$, the DFR method approximates the solution of Eq.~\eqref{eq:res0} using the following discrete loss function:
\begin{equation}\label{eq:losseq}
		\LL(u) =  \sum\limits_{k \in \J} \langle \R(u), \Phi_k \rangle_{H^{*}_\mathrm{D}\times H^1_{0,\mathrm{D}}}^2 =  \sum\limits_{k \in \J} \left( b(u, \Phi_k)-\ell(\Phi_k)\right)^2. 
\end{equation}

Equations~\eqref{eq:losseq0} and~\eqref{eq:losseq} are identical in the DFR method since its Gram matrix is the identity.

\begin{remark}[Numerical integration]
	We calculate $\LL(u)$ by numerically evaluating integrals involving the basis functions and their derivatives. Since the basis functions $\Phi_k$ in Eq.~\eqref{eq:basis} consist of products of sine and cosine functions, this naturally leads to the idea of using the Discrete Sine/Cosine Transforms (DST/DCT) as a quadrature rule~\cite{ britanak2010discrete}. Some scenarios, such as problems with low-regularity solutions, may require many integration points (and possibly non-uniformly distributed), translating to high computational costs. Nevertheless, we do not focus on designing efficient integration rules in this paper and leave this as part of our future work. 
\end{remark}

\section{The optimization method}\label{sec:optimis}

We use Adam~\cite{kingma2015ADAM} as the optimization method responsible for adjusting the hyper-parameters of our NNs. We enhance the optimizer performance calculating with a least squares (LS) approach the weights of the last hidden layer of the NN (see~\cite{ cyr2020robust, huang2006extreme} for further details of this hybrid approach).

When using the loss function~\eqref{eq:losseq} in a neural network of the form~\eqref{eq:layerspar}, one defines a quadratic system of equations with respect to the parameters $\tW_L$. Therefore, we can improve the solution by solving the following least-squares problem:
\begin{equation}\label{eq:lssys}
	\tW_L := \text{arg}\min_{\tW} \| \tA \tW -\tb  \|^2_{\ell^2},
\end{equation}
where $\tA:= [b(\N_{L\text{-}1}^i,\Phi_k)]_{i,k}$ is not necessarily a square matrix, $\N_{L\text{-}1}^i$ is the $i$-th node of the last hidden layer, and $\tb:=[\ell(\Phi_k)]_k$ is the right hand side vector.  
The solution to the least squares system~\eqref{eq:lssys} can be found by solving the normal equations. Thus, the weights of the final layer consistently correspond to optimal values, given the variable parameters $\tW_j$ and $\tb_j$ for $j \in \{1,\dots, L\text{-}1\}$, which are subsequently tuned through the gradient descent algorithm.

Figure~\ref{figArchitecture} sketches the details of the fully connected NN employed in this work. 

\begin{figure}[ht!]
	\centering
	\begin{tikzpicture}
	
	\draw (0,2.85) node{Input};
	\draw (0,2.5) node {$\overbrace{\hspace{1cm}}^{\hspace{1cm}}$} ;
	
	\draw (4,-1.5) node {Trainable layers};
	\draw (4,-1.2) node {$\underbrace{\hspace{6.7cm}}_\text{\hspace{1cm}}$} ;
	
	
	\Vertex[x=8.3,y=1,style={opacity=0.0,color=white!0},size=0.00000001]{inv1}
	\Vertex[x=8.3,y=-0.7,style={opacity=0.0,color=white!100}]{inv2}
	\draw (8.3,-0.7) node {LS};
	\Edge[Direct,lw=1pt](inv2)(inv1)
	
	\draw(11,2.85) node {Output};
	
	\draw (11,2.5) node {$\overbrace{\hspace{1cm}}^\text{\hspace{1cm}}$};
	\Vertex[x=10,y=1.5,style={opacity=0.0,color=white!0},size=0.00000001]{inv1}
	\Vertex[x=10,y=0,style={opacity=0.0,color=white!100}]{inv2}
	\draw (10,0) node {BC};
	\Edge[Direct,lw=1pt](inv2)(inv1)
	
	\draw (13,0.7) node {$\underbrace{\hspace{1cm}}^\text{\hspace{1cm}}$};
	\draw (13,0) node {Loss function};
	
	\draw[blue , dashed, line width=1pt] (13,-0.2) -- (13,-2.1) ;
	\draw[blue , dashed,line width=1pt] (13,-2.1) -- (4,-2.1);
	\draw[->,blue , dashed, line width=1pt] (4,-2.1) -- (4,-1.7);
	\draw (9,-1.8) node {{ADAM optimizer}};
	
	\Vertex[x=0,y=1.5,label=$x$,color=white,size=1,fontsize=\normalsize]{X}
	\Vertex[x=2,y=3,color=white,size=0.75]{A11}
	\Vertex[x=4,y=3,color=white,size=0.75]{A21}
	\Vertex[x=7,y=3,color=white,size=0.75]{A31}
	\Vertex[x=2,y=2,color=white,size=0.75]{A12}
	\Vertex[x=4,y=2,color=white,size=0.75]{A22}
	\Vertex[x=7,y=2,color=white,size=0.75]{A32}
	\Vertex[x=2,y=0,color=white,size=0.75]{A13}
	\Vertex[x=4,y=0,color=white,size=0.75]{A23}
	\Vertex[x=7,y=0,color=white,size=0.75]{A33}
	
	\begin{scope}[shift={(A11.center)}, scale=0.1]
		\draw[color=gray,<->,>={Stealth[length=0.5mm,width=0.5mm]}](-2,0) -- (2,0) ;
		\draw[color=gray,<->,>={Stealth[length=0.5mm,width=0.5mm]}](0,-1.5) -- (0,1.5) ;
		\draw[domain=-2:2,smooth,variable=\x,red,line width=1pt] plot ({\x},{tanh(\x)});
	\end{scope}
	\begin{scope}[shift={(A21.center)}, scale=0.1]
		\draw[color=gray,<->,>={Stealth[length=0.5mm,width=0.5mm]}](-2,0) -- (2,0) ;
		\draw[color=gray,<->,>={Stealth[length=0.5mm,width=0.5mm]}](0,-1.5) -- (0,1.5) ;
		\draw[domain=-2:2,smooth,variable=\x,red,line width=1pt] plot ({\x},{tanh(\x)});
	\end{scope}
	\begin{scope}[shift={(A31.center)}, scale=0.1]
		\draw[color=gray,<->,>={Stealth[length=0.5mm,width=0.5mm]}](-2,0) -- (2,0) ;
		\draw[color=gray,<->,>={Stealth[length=0.5mm,width=0.5mm]}](0,-1.5) -- (0,1.5) ;
		\draw[domain=-2:2,smooth,variable=\x,red,line width=1pt] plot ({\x},{tanh(\x)});
	\end{scope}
	\begin{scope}[shift={(A12.center)}, scale=0.1]
		\draw[color=gray,<->,>={Stealth[length=0.5mm,width=0.5mm]}](-2,0) -- (2,0) ;
		\draw[color=gray,<->,>={Stealth[length=0.5mm,width=0.5mm]}](0,-1.5) -- (0,1.5) ;
		\draw[domain=-2:2,smooth,variable=\x,red,line width=1pt] plot ({\x},{tanh(\x)});
	\end{scope}
	\begin{scope}[shift={(A13.center)}, scale=0.1]
		\draw[color=gray,<->,>={Stealth[length=0.5mm,width=0.5mm]}](-2,0) -- (2,0) ;
		\draw[color=gray,<->,>={Stealth[length=0.5mm,width=0.5mm]}](0,-1.5) -- (0,1.5) ;
		\draw[domain=-2:2,smooth,variable=\x,red,line width=1pt] plot ({\x},{tanh(\x)});
	\end{scope}
	\begin{scope}[shift={(A22.center)}, scale=0.1]
		\draw[color=gray,<->,>={Stealth[length=0.5mm,width=0.5mm]}](-2,0) -- (2,0) ;
		\draw[color=gray,<->,>={Stealth[length=0.5mm,width=0.5mm]}](0,-1.5) -- (0,1.5) ;
		\draw[domain=-2:2,smooth,variable=\x,red,line width=1pt] plot ({\x},{tanh(\x)});
	\end{scope}
	\begin{scope}[shift={(A23.center)}, scale=0.1]
		\draw[color=gray,<->,>={Stealth[length=0.5mm,width=0.5mm]}](-2,0) -- (2,0) ;
		\draw[color=gray,<->,>={Stealth[length=0.5mm,width=0.5mm]}](0,-1.5) -- (0,1.5) ;
		\draw[domain=-2:2,smooth,variable=\x,red,line width=1pt] plot ({\x},{tanh(\x)});
	\end{scope}
	\begin{scope}[shift={(A32.center)}, scale=0.1]
		\draw[color=gray,<->,>={Stealth[length=0.5mm,width=0.5mm]}](-2,0) -- (2,0) ;
		\draw[color=gray,<->,>={Stealth[length=0.5mm,width=0.5mm]}](0,-1.5) -- (0,1.5) ;
		\draw[domain=-2:2,smooth,variable=\x,red,line width=1pt] plot ({\x},{tanh(\x)});
	\end{scope}
	\begin{scope}[shift={(A33.center)}, scale=0.1]
		\draw[color=gray,<->,>={Stealth[length=0.5mm,width=0.5mm]}](-2,0) -- (2,0) ;
		\draw[color=gray,<->,>={Stealth[length=0.5mm,width=0.5mm]}](0,-1.5) -- (0,1.5) ;
		\draw[domain=-2:2,smooth,variable=\x,red,line width=1pt] plot ({\x},{tanh(\x)});
	\end{scope}

	\Vertex[x=9,y=1.5,label=$\tilde{u}(x)$,color=white,size=1,fontsize=\normalsize]{tildeu}
	\Vertex[x=11,y=1.5,label=$u(x)$,color=white,size=1,fontsize=\normalsize]{u}
	
	\Vertex[x=13,y=1.5,label=$\LL(u)$,color=white,size=1,fontsize=\normalsize]{lu}
	
	\Vertex[x=5.5,y=3,style={color=white},label=$\color{black}\hdots$,size=0.75]{Ah1}
	\Vertex[x=5.5,y=2,style={color=white},label=$\color{black}\hdots$,size=0.75]{Ah2}
	\Vertex[x=5.5,y=0,style={color=white},label=$\color{black}\hdots$,size=0.75]{Ah3}

	\Edge[Direct,lw=0.5pt](X)(A11)
	\Edge[Direct,lw=0.5pt](X)(A12)
	\Edge[Direct,lw=0.5pt](X)(A13)
	
	\Edge[color=white,label={$\color{black}\vdots$}](A12)(A13)
	\Edge[color=white,label={$\color{black}\vdots$}](A22)(A23)
	\Edge[color=white,label={$\color{black}\vdots$}](Ah2)(Ah3)
	\Edge[color=white,label={$\color{black}\vdots$}](A32)(A33)
	
	\Edge[Direct,lw=0.25pt](A11)(A21)
	\Edge[Direct,lw=0.25pt](A11)(A22)
	\Edge[Direct,lw=0.25pt](A11)(A23)
	\Edge[Direct,lw=0.25pt](A11)(A21)
	\Edge[Direct,lw=0.25pt](A12)(A21)
	\Edge[Direct,lw=0.25pt](A12)(A22)
	\Edge[Direct,lw=0.25pt](A12)(A23)
	\Edge[Direct,lw=0.25pt](A13)(A21)
	\Edge[Direct,lw=0.25pt](A13)(A22)
	\Edge[Direct,lw=0.25pt](A13)(A23)
	
	\Edge[Direct,lw=0.25pt](A23)(Ah1)
	\Edge[Direct,lw=0.25pt](A23)(Ah2)
	\Edge[Direct,lw=0.25pt](A23)(Ah3)
	\Edge[Direct,lw=0.25pt](A21)(Ah1)
	\Edge[Direct,lw=0.25pt](A21)(Ah2)
	\Edge[Direct,lw=0.25pt](A21)(Ah3)
	\Edge[Direct,lw=0.25pt](A23)(Ah1)
	\Edge[Direct,lw=0.25pt](A23)(Ah2)
	\Edge[Direct,lw=0.5pt](A23)(Ah3)
	
	\Edge[Direct,lw=0.5pt](Ah1)(A31)
	\Edge[Direct,lw=0.5pt](Ah1)(A32)
	\Edge[Direct,lw=0.5pt](Ah1)(A33)
	\Edge[Direct,lw=0.5pt](Ah2)(A31)
	\Edge[Direct,lw=0.5pt](Ah2)(A32)
	\Edge[Direct,lw=0.5pt](Ah2)(A33)
	\Edge[Direct,lw=0.5pt](Ah3)(A31)
	\Edge[Direct,lw=0.5pt](Ah3)(A32)
	\Edge[Direct,lw=0.5pt](Ah3)(A33)
	
	\Edge[Direct,lw=0.5pt](A31)(tildeu)
	\Edge[Direct,lw=0.5pt](A32)(tildeu)
	\Edge[Direct,lw=0.5pt](A33)(tildeu)
	
	\Edge[Direct,lw=0.5pt](tildeu)(u)
	\Edge[Direct,lw=0.5pt](u)(lu)
	
\end{tikzpicture} 
	\caption{NN architecture sketch. Non-trainable layers are:  the loss function calculation, the homogeneous Dirichlet boundary condition (BC) imposition, and the least-squares (LS) solution.}
	\label{figArchitecture}
\end{figure}
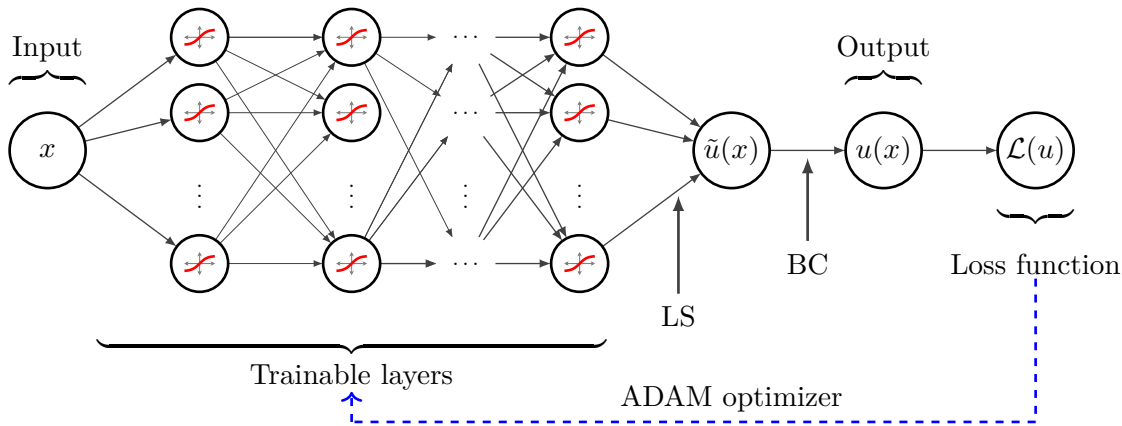

\section{The DFR-based Domain Decomposition method}\label{sec:DFRDDteo}

This section applies the DFR mathematical framework to solve problems of the form of Eq.~\eqref{eq:res0} through domain decomposition. Here, the analysis is limited to domains expressed as finite unions of $n$-rectangles, specifically, domains meeting the following assumption:
\begin{assumption}\label{as:a1}
	The open, bounded domain $\Om \subset \RR^n$ with $n=1, 2$ or $3$, is represented by a finite union of open $n$-rectangles, $\Om = \bigcup\limits_{i=1}^m\Om_{i}$ where each $\Om_{i}$ is a (possibly rotated) open $n$-rectangle.
\end{assumption}

If $\Om$ is connected, each subdomain must overlap with at least one other, and each interior point of $\Om$ is in the interior of at least one subdomain. Thus, Assumption~\ref{as:a1} describes an overlapping subdomain decomposition of $\Om$. Moreover, there exist $m_c \in \NN$ with $m_c \leq m$, such that
\[ \text{card}\{ i \, : \, x \in \Om_i \} \leq m_c \quad  \forall x \in \Om. \]

Following Section~\ref{sec:discLoss}, a possible solution strategy for the problem~\eqref{eq:res0} on a domain $\Om$, satisfying Assumption~\ref{as:a1}, is to construct a finite-dimensional subspace $\V$ within $\HoD$ as follows:
\begin{equation}
	\V = \text{span}\left( \bigcup\limits_{i=1}^m \{\Phi_k^{(i)}\}_{k \in \J^{(i)}} \right),
\end{equation}
where, for each $i \in \{1, \dots, m\}$, the set $\{\Phi_k^{(i)}\}_{k \in \J^{(i)}}$ is a chosen finite-dimensional subset of basis functions within the following space:
 \begin{equation}\label{eq:Hi}
	  H_i := \left\lbrace v \in H^1(\Om) : \text{supp}(v)\subset \bar{\Om}_i,\, \gamma_{\mathrm{D}}(v)=0\right\rbrace.
\end{equation}

Using this approach, one constructs the Gram matrix involving all the basis functions in $\V$ and employ Eq.~\eqref{eq:losseq0} as a loss function. However, a significant challenge arises due to the overlapping of the subdomains $\Om_i$. This overlapping results in the non-orthogonality of the basis functions and potentially a high condition number for the Gram matrix, leading to numerical instabilities and high computational costs. We avoid inverting the Gram matrix by extending  the DFR method using a mathematical construction that circumvents the inversion requirement. 

\begin{assumption}\label{as:a2}
	Consider a Hilbert space $H$ and a given set of closed subspaces $\{H_i\}_{i=1}^m$ within $H$. For all $v\in H$, there exists $\{v_i\}_{i=1}^m$ with $v_i \in H_i$ for all $i$, and $v=\sum\limits_{i=1}^m v_i$.
\end{assumption}

Guaranteeing Assumption~\ref{as:a2} on general domains is not trivial. Constructing such a decomposition requires careful consideration of the complexity of each geometry, including challenging features like (possibly multiple) re-entrant corners and singular points. Section~\ref{sec:extension} and Appendix~\ref{app:2}  describe an appropriate decomposition procedure for two cases of interest.

\subsection{The $\star$-norm}\label{sec:starnorm}

Now, our goal is to bound (above and below) up to a constant the dual norm of the residual $ \| {\R(u)} \| _{H^*_{\mathrm{D}}(\Om)}$ by the sum of dual norms over each of the subdomains. 
\begin{definition}
	Let $H$ be a Hilbert space with dual space $H^*$. Given a function $f\in H^*$, we define the $\star$-norm as
	\[ \left \| f\right \| _{\star} = \sqrt{\sum\limits_{i=1}^m \| f \| _{H^*_i}^2} \quad , \]
	where $\{H_i\}_{i=1}^m$ satisfies Assumption~\ref{as:a2} and $H^*_i$ denotes the dual space of $H_i$.
\end{definition}

$ \| \cdot \| _{\star}$ is an equivalent norm to the classical norm of the dual space.

\begin{proposition}\label{prop:star}
	Let $H$ be a Hilbert space with dual space $H^*$. If there exists a set of closed subspaces $\{H_i\}_{i=1}^m$ satisfying Assumption~\ref{as:a2},  then $f\in H^*$ satisfies
	\begin{equation}\label{eq:starbound}
		\frac{1}{\sqrt{m}} \| f \| _{\star} \leq \| f \| _{H^*} \leq\xi \left( \{H_i\}_{i=1}^m \right) \| f \| _{\star},
	\end{equation}
	with
	\begin{equation}
		\xi \left( \{H_i\}_{i=1}^m \right):=\sup\limits_{v\in H}\frac{1}{\|v\|_{H}}\min\left\{\sqrt{\sum\limits_{i=1}^m \|v_i\|_{H}^2}:v_i\in H_i \text{ and } v = \sum\limits_{i=1}^m v_i \right\}<+\infty. 
	\end{equation}
\end{proposition}
Appendix~\ref{app:1} detils the proof of Proposition~\ref{prop:star}.

From Eqs.~\eqref{eqUpLowerBd} and~\eqref{eq:starbound}, we estimate the error using the $\star$-norm of the residual as follows: 
\begin{equation}\label{eq:finalbound}
	\frac{1}{M \cdot \sqrt{m}} \| \R(u) \| _{\star}\leq \| u-u^* \| _{\HoD} \leq \frac{\xi \left( \{H_i\}_{i=1}^m \right)}{\gamma} \| \R(u) \| _{\star}.
\end{equation}

When constructing a subdomain decomposition $\{\Om_i\}_{i=1}^m$ satisfying Assumption~\ref{as:a1}, the overlap constant $m_c$ allows us to regard the union of non-overlapping subdomains as one. In the theory of domain decomposition methods, this \textit{coloring strategy}~\cite{badea2022convergence} allows us to sharpen the estimates~\eqref{eq:starbound} and~\eqref{eq:finalbound}.
From Eq.~\eqref{eq:finalbound}, we obtain that the ${\star}$-norm of the weak residual estimates the error, but we highlight that the quality of this estimate deteriorates as $m_c \to \infty$, making it desirable to use as few overlapping subdomains as possible to cover $\Om$.
Also, whilst our proof guarantees that $\xi$ is finite (see Appendix~\ref{app:2}), it may, in principle, be large. Providing estimations in general domains is beyond the scope of this work; however, in Theorem~\ref{thmConstantL}, we show that for the L-shaped domain, covered by two rectangular domains, the constant is at most $\sqrt{2}$.

\subsection{The DFR-DD loss function}

Applying the theory of Section~\ref{sec:dfrrect}, and considering that each subdomain $\Om_i$ is an $n$-rectangle, we can approximate the dual norm of the residual within each subdomain $\Om_i$. We use a suitable set of local basis functions to this aim, as Eq.~\eqref{eq:basis} defines. Subsequently, the local approximations of the dual norm of the weak residual are used to calculate the $\star$-norm.

Let $\Om = \bigcup\limits_{i=1}^m\Om_{i}$ where each $\Om_{i}$ is an $n$-rectangle, satisfying Assumption~\ref{as:a1}. Given $\{H_i\}_{i=1}^m$ satisfying Assumption~\ref{as:a2}, for each $i$, we choose a cut-off frequency and denote $\J^{(i)}$ the corresponding finite set of indices. The dual norm of the residual constrained to each subdomain $\Om_i$ is such that
\begin{equation}\label{eq:subresi}
	\| {\R(u)} \| _{H_i^*}^2 \approx \sum\limits_{k \in \J^{(i)}} \langle \R(u), \Phi_k^{(i)} \rangle_{H_i^*\times H_i}^2,
\end{equation}
where, for each $i$, $\{ \Phi_k^{(i)} \}_{k\in \J^{(i)}}$ is a set of orthonormal basis functions spanning a finite-dimensional subspace of $H_i$. Therefore, we compute the $\star$-norm of the residual operator as
\begin{equation}\label{eq:subresi}
	\| {\R(u)} \| _{\star}^2 = \sum\limits_{i=1}^m \| {\R(u)} \| _{H_i^*}^2 \approx \sum\limits_{i=1}^m \sum\limits_{k \in \J^{(i)}} \langle \R(u), \Phi_k^{(i)} \rangle_{H_i^*\times H_i}^2. 
\end{equation}

Finally, we minimize the following loss function:
\begin{equation}\label{eq:losseq2}
	{\LL}(u) := {\sum\limits_{i=1}^m \hat{\LL}(u ;\Om_i)}, \quad \text{ with } \quad \hat{\LL}(u;\Om_i) := \sum\limits_{k \in \J^{(i)}} \langle \R(u), \Phi_k^{(i)} \rangle_{H_i^*\times H_i}^2.
\end{equation}

The calculation of $\hat{\LL}(u;\Om_i) $ at each training iteration step is independent of the others. Therefore, the local loss functions can be computed in parallel. 

The accuracy of our DFR-based domain decomposition method relies on two crucial factors (apart from the errors introduced by numerical integration): the nature of the subdomain decomposition of $\Om$ and the truncation error associated with the number of Fourier modes (basis functions) used in the computation of Eq.~\eqref{eq:losseq2}. We elaborate on both aspects:
\begin{itemize}
	\item The accuracy of using $ \| \R(u) \| _{\star}$ as an estimate of the error $\| u-u^* \| _{\HoD}$ depends on the constants appearing in the discrete versions of~\eqref{eqUpLowerBd} and ~\eqref{eq:finalbound}. Note that, the constants bounding the norm of the error in~\eqref{eqUpLowerBd} are problem-dependent. 
	
	For instance, in Poisson's problem, the dual norm of the residual and norm of the error are equal, as $M=\gamma=1$  in~\eqref{eqUpLowerBd}. However, these constants may be highly disparate when the PDE is described by a bilinear form distinct from the canonical inner product on $H^1(\Om)$, as in the Helmholtz equation. On the other hand, the bounds in~\eqref{eqUpLowerBd} also depend on the number of overlapping subdomains used to cover $\Om$. If the subspaces $H_i$ are orthogonal, then it is straightforward to show that $\xi=1$, the $\star$-norm is precisely the dual norm. In our case of domain decompositions, the subdomains necessarily overlap, and thus, the spaces cannot be pairwise orthogonal. Therefore, we aim to compute the loss function~\eqref{eq:losseq2} using the minimum feasible number of overlapping rectangles.
	
	\item In Eq.~\eqref{eq:losseq2}, each term $ \hat{\LL}(u;\Om_i)$ approximates the norm $ \| {\R(u)} \| _{H_i^*}^2 $. For $i\in \{1, \dots, m\}$, we construct a finite set of indices $\J^{(i)}$ by choosing a cut-off frequency in each coordinate direction. This choice introduces a truncation error, which decreases only as additional Fourier modes are used in the calculation. On the other side, the computational cost of the method increases significantly as the number of Fourier modes increases. 
\end{itemize} 

\section{Applications}
\subsection{The extension of the DFR method to polygons}\label{sec:extension}

The above DFR-based domain decomposition method extends the applicability of the DFR method to domains other than the $n$-rectangles considered in~\cite{ taylor2023deep}. 

We consider open and bounded polygons (or polyhedrons) $\Om \subset \RR^n$ with $n=1, 2$ or $3$, with internal angles greater than or equal to $90$ degrees that meet Assumption~\ref{as:a1}. Figure~\ref{fig:pentagon} shows an example of a pentagon and one possible cover of rectangles $\{\Om_{i}\}_{i=1}^5$ for it.
\begin{figure}[ht!]
	\centering
	\includegraphics[width=0.95\linewidth]{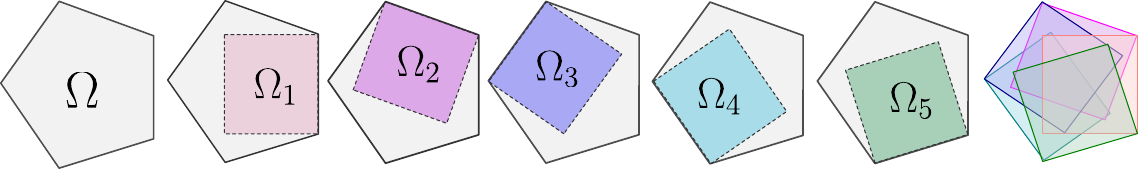}
	\caption{Example of a convex polygon $\Om$ (left), a simple cover of $\Om$ with five rotated rectangles $\Om_i$, and the superposition of all the subdomains (right). }
	\label{fig:pentagon}
\end{figure}

Assumption~\ref{as:a2}  guarantees we can apply the DFR method to rectilinear polygons. Partitions of unity are an attractive method of showing that the necessary decompositions exist; however, due to the geometric issues, the existence of sufficiently regular partitions of unity to guarantee Assumption~\ref{as:a2} holds is not immediate. In Appendix~\ref{app:2}, we provide sufficient results so that the specific cases we study can be easily addressed. In particular, we construct explicit decompositions with sufficient regularity for the L-shape and pentagonal domains, which will be used in our numerical experiments in Section~\ref{sec:num}. 

After ensuring Assumption~\ref{as:a2}, we can use Eq.~\eqref{eq:losseq2} to solve problems of the form~\eqref{eq:res0} on polygons. However, the quality of the loss function~\eqref{eq:losseq2} as an estimate of the energy norm of the error is highly related to the number of overlapping domains considered in the decomposition of $\Om$ and the choice of cut-off frequencies that approximate the dual norm of the residual locally. 

\subsection{Local refinement and adaptivity}\label{sec:localRef}

We now propose a strategy to increase the local accuracy of the DFR method. Since the subdomains $\Om_i$ in  Assumption~\ref{as:a1} may overlap, we propose a strategy that includes additional subdomains in the original decomposition of $\Om$. By doing so, we cover the potential singularities of the solution and increase the resolution of the approximation only where it is needed. 

\subsubsection{Local refinement}

Whilst the $\star$-norm is equivalent to the dual norm in certain circumstances, our capacity to evaluate it numerically is limited principally by using a finite set of spectral test functions on each subdomain. Considering a single rectangular domain, this issue would be apparent without domain decomposition. We may view the addition of higher frequency basis functions over the given subdomains as a case of $p$-refinement on a global level. However, the number of basis functions required for a particular problem, particularly in higher dimensional problems, may become numerically infeasible. As such, it would be more desirable to consider only local refinement to improve the computational costs of evaluating the loss function. We outline a heuristic argument in the following where, for simplicity, we consider only $\Hoo$ as our function space. 

Consider a set of subdomains $\Om_i$, each having a corresponding set of spectral basis functions  $\{ \Phi_k^{(i)} \}_{k\in \J^{(i)}}$, each over a finite index set $\J^{(i)}$ that may depend on $i$. Given a candidate solution $u$, as in Section~\ref{sec:discLoss}, let $\hat{u}_{\R}$ denote the Riesz representative of the residual. If the bilinear form $b$ is the canonical inner product on $\Hoo$, then $\hat{u}_{\R}=u-u^*$. Generally, however, the dependence will be non-local. Without integration errors, the component of the loss $\displaystyle \hat{\LL}(u;\Om_i)= \| Q_{\J^{(i)}} \hat{u}_R \|_{\Hoo}^2$, where $\displaystyle Q_{\J^{(i)}}:\Hoo\to \text{span}\left( \{ \Phi_k^{(i)} \}_{k\in \J^{(i)}} \right)$ is the orthogonal projection onto those basis functions.  

If $u$ has been trained so that the discretized loss function is sufficiently small, then, approximately speaking,  $\hat{u}_R$ is almost orthogonal to the entire set of basis functions $\bigcup\limits_{i=1}^m  \{ \Phi_k^{(i)} \}_{k\in \J^{(i)}}$. If $\hat{u}_R$ is sufficiently localized, in the sense that a significant fraction of its norm is contained in a subdomain that we denote $\Om_{m+1}$ with corresponding function space $H_{m+1}=H^1_0(\Om_{m+1})$, then it would be an obvious choice to include $\|\R(u)\|_{H^*_{m+1}}$ into the loss so that, in a second-stage training, these localized errors can be mitigated. Of course, we are limited again by the finiteness of our test function space, yielding a new set of basis functions $\J^{(m+1)}$ and discretized loss
\begin{equation}\label{eq:losseq3}
	{\LL}(u) := {\sum\limits_{i=1}^{m+1} \hat{\LL}(u ;\Om_i)}, \quad \text{ with } \quad \hat{\LL}(u;\Om_i) = \sum\limits_{k \in \J^{(i)}} \langle \R(u), \Phi_k^{(i)} \rangle_{H_i^*\times H_i}^2,
\end{equation} 
and
\begin{equation}
	\hat{\LL}(u;\Om_{m+1})=\sum\limits_{k\in\J^{(m+1)}}\langle \R(u), \Phi_k^{(i)}\rangle_{H^*_{m+1}\times H_{m+1}}^2=\| Q_{\J^{(m+1)}}\hat{u}_{\R}\|^2_{\Hoo}.
\end{equation}

This motivates the use of $\hat{\LL}(u;\Om_{m+1})$ as both an error indicator for localized errors of the residual -- recalling that these are, generally, not equivalent to localized errors of the solution -- as well as a new component added to the loss in a post-refinement training stage. 

However, there is a subtle issue behind this argument. The orthogonal projection $Q_{\J_i}\hat{u}_{\R}$ may be zero over a subdomain even if $\hat{u}_{\R}$ is non-zero, which will occur if, for example, $\hat{u}_{\R}$ is harmonic over the given region. Our error estimator is a local error estimator for the Riesz representative of the residual {\it modulo functions that are harmonic on the given subdomain}. If, however, $\hat{u}_{\R}$ is harmonic on the entirety of $\Om$, the total error is necessarily zero due to the imposed boundary conditions. Qualitatively, our refinement strategy is thus to identify spatial regions where the deviation of $\hat{u}_R$ from a harmonic function is large and include more localized test functions over this region into the loss to lead to a harmonic, and consequently zero, total error. 

Although the use of~\eqref{eq:losseq3} may deteriorate the correlation between the loss function and the error's energy norm, alternative strategies such as choosing Eq.~\eqref{eq:losseq} as minimization target result in a Gram matrix that is expensive to invert (or singular), rendering the RVPINN strategy infeasible and strengthening our approach. 

There are two options for refining a subdomain decomposition of $\Om$. The first and simplest is to define the location and the dimensions of the new $n$-rectangles \textit{ad-hoc}. This type of refinement is heavily influenced by prior knowledge about the problem and the solution. On the other hand, an automatic adaptive refinement of the subdomain decomposition is possible. That is, given an error indicator --in our case, the local contributions to the loss function-- one adds smaller subdomains in the decomposition only within the marked regions. This idea relates to the traditional mesh refinement techniques used in finite element methods~\cite{ babuvvska1978error,paszynski2021deep} and to recent refinement strategies using dual norms~\cite{ calo2019, rojas2021, cier2021}. Adding new subdomains to the initial domain decomposition of $\Om$ increases the accuracy of the approximation, as in $h$-refinement techniques (see~\cite{ babuvska11984performance,demkowicz1985h}). Conversely, one can add Fourier modes to specific subdomains, mimicking $p$-refinement techniques (see~\cite{ babuska1981p,babuvska11984performance}). Hence, our methodology may be placed within the context of both $h$- and $p$-refinements in the test space. 

\subsubsection{$h$-adaptive local refinement}\label{sec:adaptive}

We propose an automatic adaptive strategy based on using the individual contributions of the terms $\hat{\LL}(u;\Om_i)$ in Eqs.~\eqref{eq:losseq2} and~\eqref{eq:losseq3} as local error indicators. 

\subsubsection*{Subdomain decomposition of $\Om$}
Consider a conforming partition of $\Om$ into $n$-rectangular subdomains, denoted as $\Th$, where the diameter of the subdomains is bounded by $h$. The inclusion of inner vertices in $\Th$, and for each inner vertex $x_i$, we can construct the classical \textit{hat functions} $\nu_i$, which satisfy $\nu_i( x_j) = \delta_{ij}$ for all $i,j \in \{1, \dots, m\}$. Then, an initial decomposition of $\Om$ consists of the following subdomains: \[\Om_{i} = \text{supp}(\nu_i) \quad \text{ with } i \in \{1, \dots, m\}.\]

Figure~\ref{fig:ex1ab} sketches an initial subdomain decomposition of a 1D domain, including the vertices, the hat functions, and the obtained subdomains $\Om_i$, with $i \in \{1, \dots, 5\}$. 
\begin{figure}[ht!]
	\centering
	\begin{tikzpicture}[scale=0.9]
	\begin{axis}[
		width=0.8\textwidth,
		height=4cm,
		xmin=-0.5,
		xmax=10.5,
		ymin=-0.1,
		ymax=1.1,
		grid = both,
		grid style={dashed},
		axis background/.style={fill=gray!5},
		xtick={0,2,4,6,8,10},
		xticklabels={$a$, $x_1$, $x_2$, $x_3$, $x_4$, $b$},
		ytick={0,1},
		smooth,
		]
		
		\addplot [cyan, line width=1.2pt] coordinates {(0,0) (2,1)};
		\addplot [cyan, line width=1.2pt] coordinates {(2,1) (4,0)};
		\addplot [green!70!black, line width=1.2pt] coordinates {(2,0) (4,1)};
		\addplot [green!70!black, line width=1.2pt] coordinates {(4,1) (6,0)};
		\addplot [magenta!80!black, line width=1.2pt] coordinates {(4,0) (6,1)};
		\addplot [magenta!80!black, line width=1.2pt] coordinates {(6,1) (8,0)};
		\addplot [cyan, line width=1.2pt] coordinates {(6,0) (8,1)};
		\addplot [cyan, line width=1.2pt] coordinates {(8,1) (10,0)};
		
		\addplot [only marks, mark=*] coordinates {(0,0) (2,0) (4,0) (6,0) (8,0) (10,0)};
		
		\node[] at (axis cs: 1.3,0.85) {\textcolor{cyan}{$\nu_1$}};
		\node[] at (axis cs: 3.3,0.85) {\textcolor{green!70!black}{$\nu_2$}};
		\node[] at (axis cs: 5.3,0.85) {\textcolor{magenta!80!black}{$\nu_3$}};
		\node[] at (axis cs: 7.3,0.85) {\textcolor{cyan}{$\nu_4$}};
	\end{axis}
	
	\draw [cyan,decorate,decoration={brace,amplitude=5pt,mirror}]
	(0.5,-0.5) -- (4.8,-0.5) node[black,midway,below,yshift=-5pt] {$\Om_{1}$};
	
	\draw [green!60!black,decorate,decoration={brace,amplitude=5pt,mirror=false}]
	(2.65,2.7) -- (6.8,2.7) node[black,midway,above,yshift=5pt] {$\Om_{2}$};
	
	\draw [magenta!80!black,decorate,decoration={brace,amplitude=5pt,mirror}]
	(4.8,-0.5) -- (9,-0.5) node[black,midway,below,yshift=-5pt] {$\Om_{3}$};
	
	\draw [cyan,decorate,decoration={brace,amplitude=5pt,mirror=false}]
	(6.8,2.7) -- (11,2.7) node[black, midway,above,yshift=5pt] {$\Om_{4}$};
\end{tikzpicture} 
	\vspace{-0.3cm}
	\caption{Example of an initial partition of the domain $\Om=(a,b)$, the hat functions $\nu_i$ and the resulting subdomains $\Om_{i}$. }
	\label{fig:ex1ab}
\end{figure}
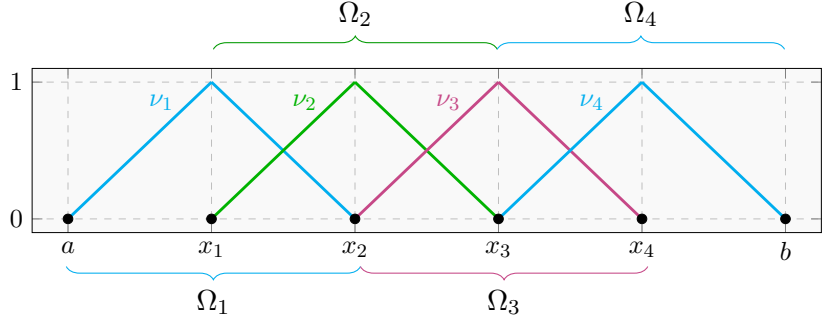

\subsubsection*{Refinement process}

Let $\{\Om_i\}_{i=1}^m$ be an initial subdomain decomposition of $\Om$ satisfying Assumption~\ref{as:a1}. To refine the subdomain decomposition $\{\Om_i\}_{i=1}^m$, we propose an iterative loop consisting of three steps:

\begin{enumerate}
	\item Train a NN with the loss function~\eqref{eq:losseq2} until achieving a prescribed stopping criteria.
	
	Here, the stopping criteria can be a prescribed maximum number of iterations, the difference between the loss of the training and validation sets, or a threshold on the loss value. 
	
	\item For $i\in \{1, \dots, m\}$, propose a subdomain decomposition of $\Om_i$, namely $\{\Om_{i,j}\}_{j=1}^{m'}$.
	
	\begin{figure}[ht!]
		\centering
		\begin{tikzpicture}[scale=0.9]
	\begin{axis}[
		width=0.8\textwidth,
		height=4cm,
		xmin=-0.5,
		xmax=10.5,
		ymin=-0.1,
		ymax=1.1,
		grid = both,
		grid style={dashed},
		axis background/.style={fill=gray!5},
		xtick={0,4,5,6,7,8,10},
		xticklabels={$a$, $x_1'$, $x_2'$, $x_3'$, $x_4'$, $x_5'$,$b$},
		ytick={0,1},
		smooth,
		]
		
		\addplot [gray, line width=0.5pt] coordinates {(0,0) (2,1)};
		\addplot [gray, line width=0.5pt] coordinates {(2,1) (4,0)};
		\addplot [gray, line width=0.5pt] coordinates {(2,0) (4,1)};
		\addplot [gray, line width=0.5pt] coordinates {(4,1) (6,0)};
		\addplot [gray, line width=0.5pt] coordinates {(4,0) (6,1)};
		\addplot [gray, line width=0.5pt] coordinates {(6,1) (8,0)};
		\addplot [gray, line width=0.5pt] coordinates {(6,0) (8,1)};
		\addplot [gray, line width=0.5pt] coordinates {(8,1) (10,0)};
		
		\addplot [cyan, line width=1.5pt] coordinates {(4,0) (5,1)};
		\addplot [cyan, line width=1.5pt] coordinates {(5,1) (6,0)};
		\addplot [magenta, line width=1.5pt] coordinates {(5,0) (6,1)};
		\addplot [magenta, line width=1.5pt] coordinates {(6,1) (7,0)};
		\addplot [green!70!black, line width=1.5pt] coordinates {(7,1) (8,0)};
		\addplot [green!70!black, line width=1.5pt] coordinates {(6,0) (7,1)};
		
		\addplot [gray, only marks, mark=*] coordinates { (4,0) (6,0) (8,0) (5,0) (7,0)};
		
	\end{axis}
	
	\draw [cyan,decorate,decoration={brace,amplitude=5pt,mirror}]
	(4.8,-0.6) -- (6.8,-0.6) node[black,midway,below,yshift=-5pt] {$\Om_{3,1}$};
	
	\draw [magenta,decorate,decoration={brace,amplitude=5pt,mirror=false}]
	(5.8,2.7) -- (7.8,2.7) node[black,midway,above,yshift=5pt] {$\Om_{3,2}$};
	
	\draw [green!70!black,decorate,decoration={brace,amplitude=5pt,mirror}]
	(6.8,-0.6) -- (8.8,-0.6) node[black,midway,below,yshift=-5pt] {$\Om_{3,3}$};
	
\end{tikzpicture}
		\vspace{-0.3cm}
		\caption{Refinement process sketch (1D). Example of a proposed decomposition of the subdomain $\Om_{3}$.}
		\label{fig:refin} 
	\end{figure}
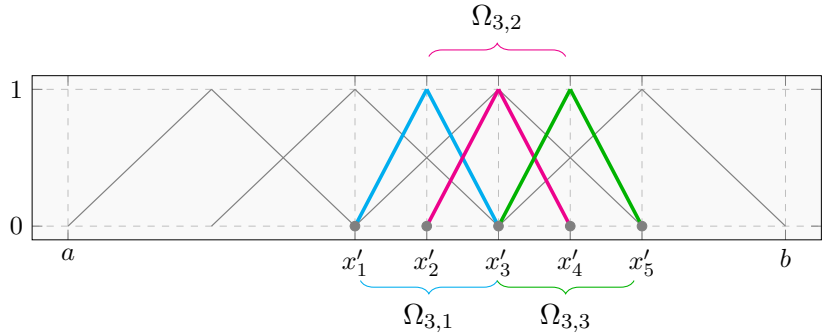
	Figure~\ref{fig:refin} shows the hat functions corresponding to a potential partition of the subdomain $\Om_3$ in Figure~\ref{fig:ex1ab}. Various strategies can partition a subdomain. However, our implementation prioritizes partitions with a minimal number of fully overlapping subdomains to control the bounds in Eq.~\eqref{eq:starbound}. 
	
	\item For $i\in \{1, \dots, m\}$ and $j\in \{1, \dots, m'\}$, add the subdomain $\Om_{i,j}$ to the previous decomposition, if it satisfies:
	\[\hat{\LL}(u;{\Om}_{i,j}) > \tau \left( \max\limits_{r,s} \hat{\LL}(u;{\Om}_{r,s}) \right),\]
	where $\tau \in (0,1]$ is a fixed threshold. 
\end{enumerate}

Algorithm~\hyperref[Algo1]{1} outlines the adaptive local refinement implementation described in  this section. 
\RestyleAlgo{ruled}
\begin{algorithm}[h!]\label{Algo1}
	\caption{Automatic adaptive local refinement}\label{alg:three}
	
	\KwData{cut-off frequencies, $\tau \in (0,1]$ and $\max_\mathrm{ref}>0$}
	
	Choose an initial subdomain decomposition of $\Om$ and let $q = 0$\
	
	\While{$q \leq \max_\mathrm{ref}$}{
		Perform the optimization of the loss function $\LL(u)$ \eqref{eq:losseq2}\
		
		Generate a subdomain decomposition of each $\Om_i$\
		
		Calculate the local error indicators $\varepsilon_{i,j} = \hat{\LL}(u;\Om_{i,j})$\ 
		
		\If{$\varepsilon_{i,j} > \tau \left(\max\limits_{r,s} \varepsilon_{r,s}\right)$}{
			Include subdomain $\Om_{i,j}$ in the decomposition of $\Om$\
		}{}
		$q = q + 1$\
	}
\end{algorithm}

\begin{remark}
	Figure~\ref{fig:sub2D} illustrates an initial subdomain decomposition of a polygonal domain in 2D. There, we show a simple partition of a rectilinear polynomial, the contour plot of the hat functions in 2D, and the resulting subdomain decomposition of $\Om$. 
	\begin{figure}[ht!]
		\centering
		\begin{subfigure}[t]{0.25\textwidth}
			\centering
			\includegraphics[width=0.55\textwidth]{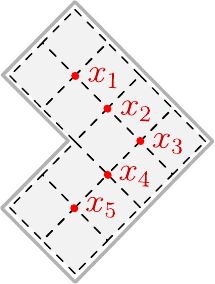}
			\caption{Domain partition.}
			\label{fig:sub2Da}
		\end{subfigure}
		\begin{subfigure}[t]{0.25\textwidth}
			\centering
			\includegraphics[width=1\textwidth]{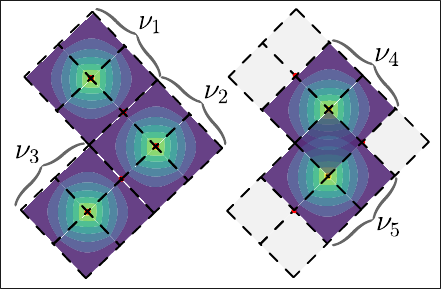}
			\caption{Hat functions in 2D.}
			\label{fig:sub2Db}
		\end{subfigure}
		\begin{subfigure}[t]{0.25\textwidth}
			\centering
			\includegraphics[width=1\textwidth]{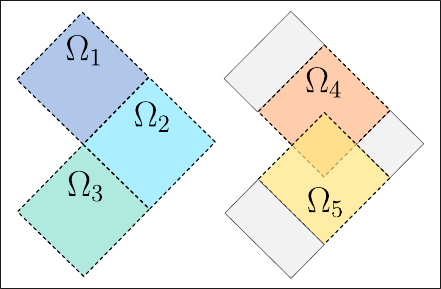}
			\caption{Subdomain decomposition.}
			\label{fig:sub2Dc}
		\end{subfigure}
		\caption{Example of an initial partition of a rectilinear polygon, the hat functions $\nu_i$ in 2D and the resulting subdomain decomposition.}
		\label{fig:sub2D}
	\end{figure}
	The process for automatic local refinements in 2D and 3D is similar to that in the 1D case. In higher dimensions, our algorithm results in quad- and oct-trees that can efficiently and dynamically store different levels of refinement. More information about these hierarchical mesh structures can be found in~\cite{ har2011geometric,mehta2004handbook}. Our implementation is currently limited to 1D cases and is only a proof of concept for 2D and 3D scenarios. Moreover, the number of subdomains increases significantly between refinement levels in 2D and even more so in 3D, and in these cases, the lack of orthogonality of the basis functions introduces new complications.
\end{remark}

\section{Numerical experiments}\label{sec:num}
This section shows numerical experiments in 1D and 2D that illustrate the potential of the proposed adaptive DFR technique. First, we show the relevance of using a subdomain decomposition for solving problems over 2D polygonal domains. We use pentagonal and classical L-shape domains, in which the traditional DFR method cannot be used. Later, we again use the L-shape domain to show the improvements in the solution in 2D given by a refined subdomain decomposition of the L-shape (\textit{ad-hoc}). Finally, we illustrate the performance of the automatic mesh-adaptive algorithm on 1D problems, some exhibiting low-regularity solutions.

In all the experiments, we use \textit{Tensorflow 2.10} and implement a feed-forward fully connected NN consisting of three hidden layers. The first two layers contain ten nodes, while the final layer consists of twenty nodes, resulting in approximately $370$ trainable variables. Every neuron in the network uses the hyperbolic tangent ($\tanh$) activation function. Moreover, we utilize the hybrid Least squares (LS)-Adam optimizer with a different initial learning rate in each case.

In all cases, we strategically distribute the integration points to mitigate the influence of numerical artifacts due to integration, which are significant in situations where the solution's gradient has singularities. Our methodology includes the results from a validation set, employing the same test functions with a significantly larger set of integration points, to monitor and address integration-related errors (as detailed in~\cite{rivera2022quadrature}). During the validation phase, we always employ a distribution of integration points $117\%$ denser than the one used during training\footnote{The Python scripts that generate these numerical results are available in the supplementary  \href{https://github.com/Mathmode/Adaptive-Deep-Fourier-Residual-method-via-overlapping-domain-decomposition/tree/main.}{GitHub repository}}.

\subsection{Case 1. Smooth solution in 2D}\label{sec:case0}
Consider the pentagonal domain $\Om= \Om_1 \cup \Om_2$, with $\Om_1 = (-1,1)^2$ and $\Om_2$ being the rotated square with vertices  $(1,-1)$, $(2,0)$, $(1,1)$ and  $(0,0)$. This geometry is shown in Figure~\ref{figCase0om}. 

We seek for the solution of the following PDE in variational form: 
\begin{equation}\label{eq:Lshape}
	\int_\Om \nabla u \cdot \nabla v \, + f v \, \dx = 0 \qquad \forall v\in \Hoo,
\end{equation}
where $f\in L^2(\Om)$ is such that the exact solution is
\begin{equation}
	u^*(x,y)= (x^2 + y^2 - 0.25)s(x,y),
\end{equation}
and $s$ is the following cut-off function:
\begin{equation}
	s(x,y)= (x + 1) (1 - y^2) ((x - 1) -( y + 1)) ((y - 1) + (x - 1)).
\end{equation} 

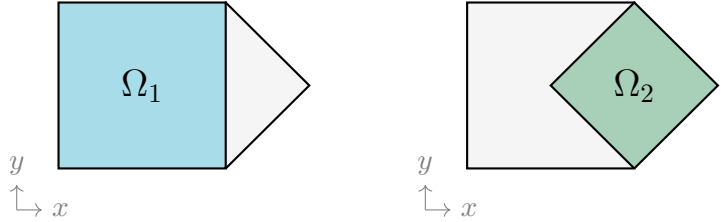
\begin{figure}[ht!]
	\centering
	\begin{subfigure}[t]{0.32\textwidth}
		\centering
		\begin{tikzpicture}[scale=1.1]
			\begin{scope}
				\draw[ thick,fill=gray!8] (1.5,-0.5)--(2.5,0.5)--(1.5,1.5)--(0.5,0.5)--(1.5,-0.5);
				\draw[fill=mycolor1, thick](-0.5,-0.5) rectangle (1.5,1.5);
				\node (a) at (0.5,0.5) {\Large $\Om_1$};
				{\draw[->,gray] (-1,-1) -- (-0.7,-1) node[right] {$x$};
					\draw[->,gray] (-1,-1) -- (-1,-0.7) node[above] {$y$};}
			\end{scope}
		\end{tikzpicture}
	\end{subfigure}
	\begin{subfigure}[t]{0.32\textwidth}
		\centering
		\begin{tikzpicture}[scale=1.1]
			\begin{scope}
				\draw[ thick,fill=gray!8](-0.5,-0.5) rectangle (1.5,1.5);
				\draw[fill=mycolor2, thick] (1.5,-0.5)--(2.5,0.5)--(1.5,1.5)--(0.5,0.5)--(1.5,-0.5);
				\node (a) at (1.5 ,0.5) {\Large $\Om_2$};
				{ \draw[->,gray] (-1,-1) -- (-0.7,-1) node[right] {$x$};
					\draw[->,gray] (-1,-1) -- (-1,-0.7) node[above] {$y$};}
			\end{scope}
		\end{tikzpicture}
	\end{subfigure}
	\caption{The cover of the pentagonal domain.}
	\label{figCase0om}
\end{figure}

Clearly, the construction of $\Om$ as the union of $\Om_1$ and $\Om_2$ satisfies Assumption~\ref{as:a1}. Furthermore, Theorem~\ref{thm:Lshapepent_app} ensures that the set $\{ H^1_0(\Om_1), H^1_0(\Om_2)\}$ satisfies Assumption~\ref{as:a2}. Hence, we can use the DFR-based domain decomposition theory to solve Eq.~\eqref{eq:Lshape} on the pentagonal domain. 

We set the learning rate of the LS-Adam optimizer to be $10^{-2}$ and employ $10 \times 10$ Fourier modes and $100 \times 100$ uniformly distributed integration points on both subdomains. Moreover, the relative $\Hoo$-norm of the error is computed using the midpoint rule, employing a grid consisting of $300$ points along both the $x$ and $y$ axes, spanning the domain $(-1,2)\times(-1,1)$. 

Figure~\ref{fig:SolutionEx0tot} presents the results of the DFR method applied to the pentagonal domain. Figure~\ref{fig:SolutionEx0} shows the reference solution $u^*$, the approximate solution $u$, and the pointwise error. Figure~\ref{fig:GradSolutionEx0} provides the magnitude of the gradient of the reference solution, the approximate solution and the magnitude of the error in the gradient. 
\begin{figure}[ht!]
	\begin{subfigure}[b]{1\textwidth}
		\centering
		\includegraphics{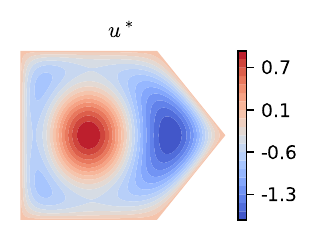}
		\hspace{-0.5cm}
		\includegraphics{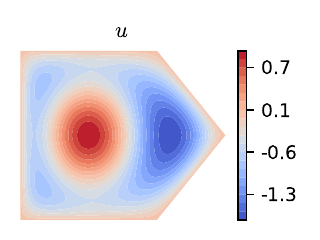}
		\hspace{-0.5cm}
		\includegraphics{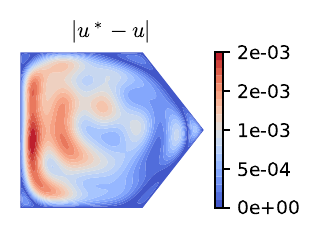}
		\vspace{-0.3cm}
		\caption{Exact solution (left), approximate solution (centre) and error in the solution (right).}
		\label{fig:SolutionEx0}
	\end{subfigure}
	\begin{subfigure}[b]{1\textwidth}
		\centering
		\includegraphics{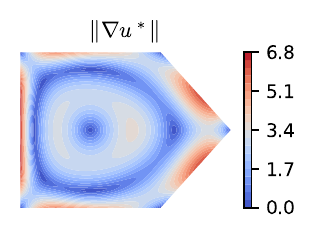}
		\hspace{-0.5cm}
		\includegraphics{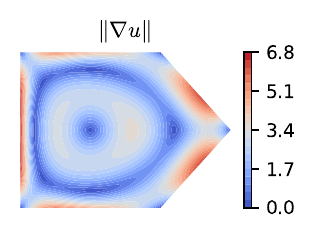}
		\hspace{-0.5cm}
		\includegraphics{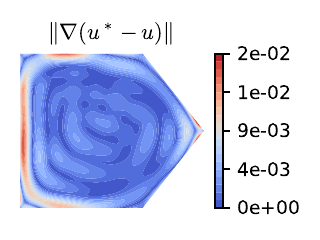}
		\vspace{-0.3cm}
		\caption{Gradient of the exact solution (left), gradient of the approximate solution (centre) and error in the gradient of the solution (right).}
		\label{fig:GradSolutionEx0}
	\end{subfigure}
	\caption{The solution and the gradient of the solution for the model  Case~\hyperref[sec:case0]{1}.}
	\label{fig:SolutionEx0tot}
\end{figure}

Figure~\ref{fig:ex0_lossanderror} shows the evolution of the loss and the percentage of the $\Hoo$-error.  The final relative $\Hoo$-error is $0.13\%$ after $5000$ iterations. This outcome represents an encouraging advance in addressing non-rectangular problems using a spectral-type approach. Moreover, Figure~\ref{fig:ex0_losserror} shows the linear correlation between the loss and the error. One observes a linear correlation between the loss and the error, as indicated by the reference dash line of slope $1$. This linear correlation was not anticipated due to the non-orthogonality of the test basis functions.
\begin{figure}[ht!]
	\centering
	\begin{subfigure}[b]{0.45\textwidth}
		\centering
		\includegraphics{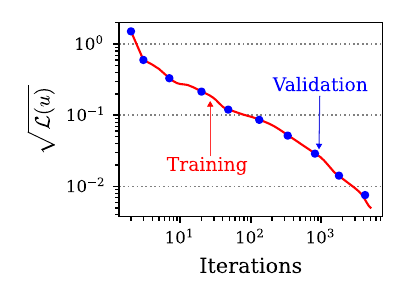}
		\vspace{-0.5cm}
		\caption{The evolution of the loss function on both the training and validation.}
		\label{fig:ex0_lossall}
	\end{subfigure}
	\hspace{0.5cm}
	\begin{subfigure}[b]{0.45\textwidth}
		\centering
		\includegraphics{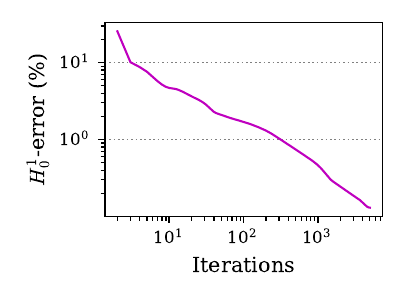}
		\vspace{-0.5cm}
		\caption{The evolution of the relative $\Hoo$-norm of the error. }
		\label{fig:ex0_errorall}
	\end{subfigure}
	\begin{subfigure}[b]{0.5\textwidth}
		\centering
		\includegraphics{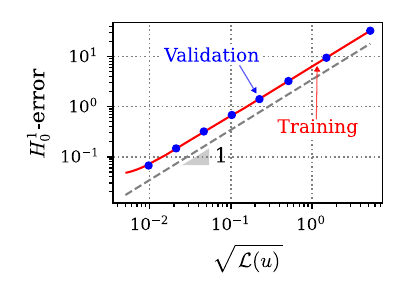}
		\vspace{-0.5cm}
		\caption{The correlation between the loss and the error.}
		\label{fig:ex0_losserror}
	\end{subfigure}
	\caption{The loss and the error for the model Case~\hyperref[sec:case0]{1}.}
	\label{fig:ex0_lossanderror}
\end{figure}

\subsection{Case 2. The L-shape domain}\label{sec:case1}

Consider the L-shape domain, $\Om=(-1,1)^2\setminus [-1,0]^2$, with Dirichlet boundary conditions on $\partial \Om$. We seek for the solution of Eq.~\eqref{eq:Lshape}, where $f\in L^2(\Om)$ is such that the exact solution is
\begin{equation}
	u^*(x,y)=(x^2-1)(y^2-1)s,
\end{equation}
and $s$ is the following singular solution in polar coordinates centred at the origin:
\begin{equation}
	s(r,\theta)=r^\frac{2}{3}\sin\left(\frac{2}{3}\left(\theta-\pi\right)\right).
\end{equation}

We cover $\Om$ with two rectangular domains, $\Om_1=(-1,0)\times(-1,1)$ and $\Om_2=(0,1)\times(-1,0)$ satisfying Assumption~\ref{as:a1}, as shown in Figure~\ref{figCase1Lshape}. 
\begin{figure}[ht!]
	\centering
	\begin{subfigure}[t]{0.32\textwidth}
		\centering
		\begin{tikzpicture}[scale=1.3]
			\begin{scope}
				\draw[fill=mycolor1] (0,0) -- (-1,0) -- (-1,1) -- 
				(1,1) -- (1,-1) -- (0,-1) -- (0,0);
				\draw[fill=gray!8] (0,-1)--(1,-1)--(1,0)--(0,0)--(0,-1);
				\node (a) at (0,0.5) {\Large $\Om_1$};
				{ \draw[->,gray] (-1,-1) -- (-0.7,-1) node[right] {$x$};
					\draw[->,gray] (-1,-1) -- (-1,-0.7) node[above] {$y$};}
			\end{scope}
		\end{tikzpicture}
	\end{subfigure}
	\begin{subfigure}[t]{0.32\textwidth}
		\centering
		\begin{tikzpicture}[scale=1.3]
			\begin{scope}
				\draw[fill=mycolor2] (0,0) -- (-1,0) -- (-1,1) -- 
				(1,1) -- (1,-1) -- (0,-1) -- (0,0);
				\draw[fill=gray!8] (-1,0)--(-1,1)--(0,1)--(0,0)--(-1,0);
				\node (a) at (0.5,0) {\Large $\Om_2$};
				{ \draw[->,gray] (-1,-1) -- (-0.7,-1) node[right] {$x$};
					\draw[->,gray] (-1,-1) -- (-1,-0.7) node[above] {$y$};}
			\end{scope}
		\end{tikzpicture}
	\end{subfigure}
	\caption{The cover of the L-shape domain.}
	\label{figCase1Lshape}
\end{figure}
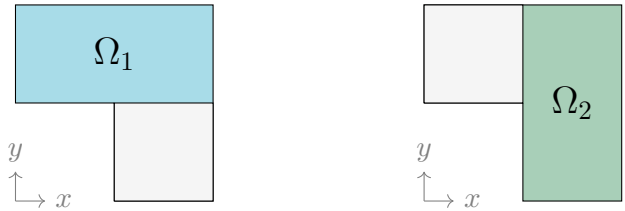

In this specific case, Theorem~\ref{thm:Lshapepent_app} ensures that $\{ H^1_0(\Om_1), H^1_0(\Om_2) \}$ satisfies Assumption~\ref{as:a2}. Moreover, in Appendix~\ref{app:3} we give a more refined analysis of the equivalence constants for the case of the L-shape domain and obtain an upper bound $\xi = \sqrt{2}$. Knowing the equivalence constants and given that the exact solution $u^*$ is singular, we use this example to 
demonstrate the effect of truncation errors in the test function space, related to higher frequency components in the residual. First, Figure~\ref{fig:SolutionEx1tot} illustrates our reference results of the DFR method on the L-shape domain. For this, we set the learning rate to be $10^{-2}$ and use $20 \times 10$ Fourier modes and $500 \times 250$ integration points on $\Om_1$, and $10 \times 20$ Fourier modes and $250 \times 500$ integration points on $\Om_2$. In both cases, the integration points are strategically distributed in a geometric pattern around the singularity at the re-entrant corner. Moreover, here the relative $\Hoo$-norm of the error is computed using an integration rule with $500 \times 500$ points spanning the domain $(-1,1)\times(-1,1)$ and also geometrically distributed around $(0,0)$. 
\begin{figure}[ht!]
	\begin{subfigure}[b]{1\textwidth}
		\centering
		\includegraphics{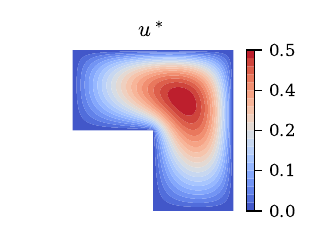}
		\hspace{-1cm}
		\includegraphics{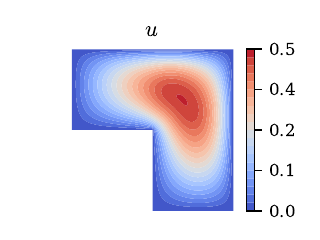}
		\hspace{-0.5cm}
		\includegraphics{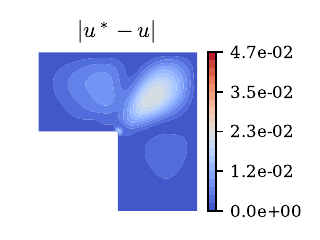}
		\vspace{-0.3cm}
		\caption{Exact solution (left), approximate solution (centre) and error in the solution (right).}
		\label{fig:SolutionEx1}
	\end{subfigure}
	\begin{subfigure}[b]{1\textwidth}
		\centering
		\includegraphics{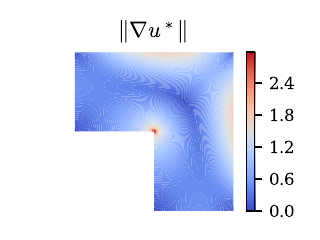}
		\hspace{-1cm}
		\includegraphics{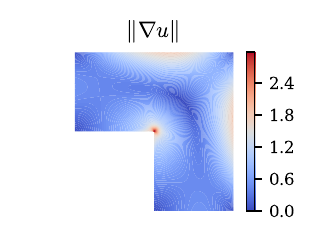}
		\hspace{-0.5cm}
		\includegraphics{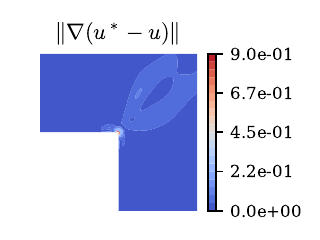}
		\vspace{-0.3cm}
		\caption{Gradient of the exact solution (left), gradient of the approximate solution (centre) and error in the gradient of the solution (right).}
		\label{fig:GradSolutionEx1}
	\end{subfigure}
	\caption{The solution and the gradient of the solution for the model Case~\hyperref[sec:case1]{2}.}
	\label{fig:SolutionEx1tot}
\end{figure}

Figure~\ref{fig:ex1_lossanderror} shows the evolution of the loss function and $\Hoo$-norm of the error for a case using $20$ modes on the long axis. The relative $\Hoo$-error reaches $2.79\%$ after $1000$ iterations.  

 Figure~\ref{fig:ex1_lossvserror} shows the correlation between the loss and the $\Hoo$-error, showing the curves corresponding to $10$, $20$ (the reference) and $60$ Fourier modes on the long axis and half as many along the short axis in each case, i.e., a total of $100$, $400$ and $3600$ test functions, respectively. There, we also provide a correlation line whose width corresponds to the known constants of equivalence $\frac{1}{\sqrt{2}}$ and $\sqrt{2}$. In the reference case, with $400$ test functions, the truncation error dominates asymptotically after an absolute $H_0^1$-error of $0.35$, and this behavior is less prevalent when increasing the number of test functions. We refer to~\cite{ rojas2023robust,taylor2023deep} for a more detailed discussion of this phenomenon. The authors explained the implications of the low quality of the chosen finite-dimensional subspaces of $\Hoo$. This is a motivation to improve the approximated solution by increasing the number of test functions near the singularity, which will help to reduce errors and improve the loss/error ratio.

\begin{figure}[ht!]
	\centering
	\begin{subfigure}[b]{0.45\textwidth}
		\centering
		\includegraphics{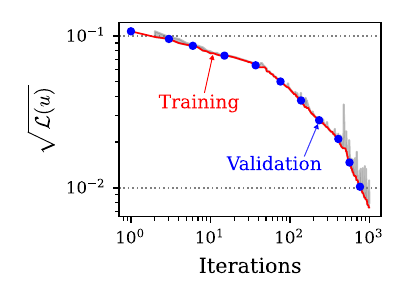}
		\vspace{-0.5cm}
		\caption{The evolution of the loss function on both the training and validation.}
		\label{fig:ex1_lossall}
	\end{subfigure}
	\hspace{0.5cm}
	\begin{subfigure}[b]{0.45\textwidth}
		\centering
		\includegraphics{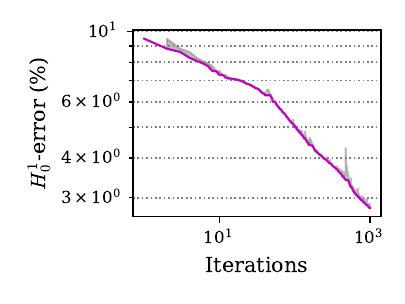}
		\vspace{-0.5cm}
		\caption{The evolution of the relative  $H_0^1$-error in percentage. }
		\label{fig:ex1_errorall}
	\end{subfigure}
		\begin{subfigure}[b]{0.5\textwidth}
		\centering
		\includegraphics{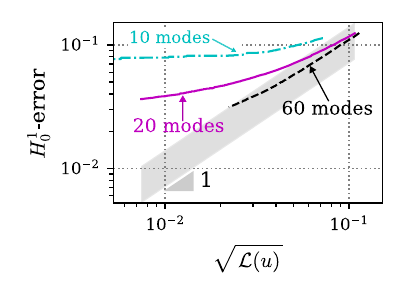}
		\vspace{-0.5cm}
		\caption{The correlation between the loss and the absolute $H_0^1$-error.}
		\label{fig:ex1_lossvserror}
	\end{subfigure}
	\caption{The loss and the error for the model Case~\hyperref[sec:case1]{2}.}
	\label{fig:ex1_lossanderror}
\end{figure}

\subsection{Case 3. The L-shape with local refinements (ad-hoc)}\label{sec:case3}

The L-shape in Case~\hyperref[sec:case1]{2} is distinctive because, as previously noted, the solution exhibits a singular derivative near the re-entrant corner, even if the source term $f$ in Eq.~\eqref{eq:Lshape} is smooth (see~\cite{ciarlet2002finite}). The reduced regularity is due to the presence of the corner itself, which introduces additional complexity to the problem. Therefore, this is an ideal candidate for mesh refinements. Since we know the singularity's location precisely, we can place smaller subdomains near the singularity without increasing the number of Fourier modes over the entire domain.

Starting from the subdomains $\Om_1$ and $\Om_2$ from Figure~\ref{figCase1Lshape}, we refine the partition by sequentially introducing four new overlapping subdomains near the re-entrant corner. The \textit{a-priori}-designed subdomains are shown in Figure~\ref{fig:refinementlshape}.
\begin{figure}[ht!]
	\centering
	\begin{minipage}{0.2\textwidth}
		\begin{equation}
			\begin{split}
				\Om_3 =&(-0.6,0.6)\times(0,0.6),\\
				\Om_4 =&(0,0.6)\times(-0.6,0.6),\\
				\Om_5 =&(-0.2,0.2)\times(0,0.2),\\
				\Om_6 =&(0,0.2)\times(-0.2,0.2).
			\end{split}
		\end{equation}
	\end{minipage}
	\hspace{1cm}
	\begin{minipage}{0.6\textwidth}
		\begin{center}
			\begin{subfigure}[t]{0.32\textwidth}
				\centering
				\begin{tikzpicture}[scale=1.3]
					\begin{scope}
						\draw[fill=gray!8] (0,0) -- (-1,0) -- (-1,1) -- 
						(1,1) -- (1,-1) -- (0,-1) -- (0,0);
						\draw[fill=mycolor1,dashed] (-0.6,0)--(0.6,0)--(0.6,0.6) -- (-0.6,0.6) -- (-0.6,0);
						\draw[fill=mycolor3,dashed] (-0.2,0)--(0.2,0)--(0.2,0.2) -- (-0.2,0.2) -- (-0.2,0);
						\node[] at (0,0.77) {\textcolor{black}{$\Om_3$}};
						\node[] at (-0.37,0.12) {\textcolor{black}{{\scriptsize  $\Om_5$}}};
						{ \draw[->,gray] (-1,-1) -- (-0.7,-1) node[right] {$x$};
							\draw[->,gray] (-1,-1) -- (-1,-0.7) node[above] {$y$};}
					\end{scope}
				\end{tikzpicture}
			\end{subfigure}
			\begin{subfigure}[t]{0.32\textwidth}
				\centering
				\begin{tikzpicture}[scale=1.3]
					\begin{scope}
						\draw[fill=gray!8] (0,0) -- (-1,0) -- (-1,1) -- 
						(1,1) -- (1,-1) -- (0,-1) -- (0,0);
						\draw[fill=mycolor2,dashed] (0,-0.6)--(0.6,-0.6)--(0.6,0.6) -- (0,0.6) -- (0,-0.6);
						\draw[fill=mycolor4,dashed] (0,-0.2)--(0.2,-0.2)--(0.2,0.2) -- (0,0.2) -- (0,-0.2);
						
						\node[] at (0.8,0) {\textcolor{black}{$\Om_4$}};
						\node[] at (0.17,-0.37) {\textcolor{black}{{\scriptsize$\Om_6$}}};
						{ \draw[->,gray] (-1,-1) -- (-0.7,-1) node[right] {$x$};
							\draw[->,gray] (-1,-1) -- (-1,-0.7) node[above] {$y$};}
					\end{scope}
				\end{tikzpicture}
			\end{subfigure}
		\end{center}
		\vspace{-0.5cm}
		\caption{The local refinement of $\Om$ with new subdomains near the singularity.} \label{fig:refinementlshape}
	\end{minipage}
\end{figure}

We set a learning rate of $10^{-2}$, and for each subdomain, we use $20 \times 10$ Fourier modes, as in Case~\hyperref[sec:case1]{2}. Regarding numerical integration, we maintain a quadrature rule in which the integration points are geometrically distributed around the re-entrant corner. Specifically, we utilize $500 \times 250$ integration points on $\Om_{1}$ and the same number on $\Om_{2}$. To integrate within the subdomains $\Om_3$, $\Om_4$, $\Om_5$, and $\Om_6$, we use the integration points from the larger domains within each subdomain. We choose this to ensure that the approximate solution $u$ is integrated on the same scale, even when different domains are involved. Moreover, for the $\Hoo$-error calculation, we again use $500 \times 500$ geometrically distributed points spanning the domain.

We begin training with the domains $\Om_1$ and $\Om_2$. After $1000$ iterations, we incorporate the domains $\Om_3$ and $\Om_4$ into the decomposition, and following $1000$ iterations, we include $\Om_5$ and $\Om_6$. In the final refinement, the total number of test functions is $1200$.

In Figure~\ref{fig:ex1_ref_lossanderror}, we show the evolution of the loss and the relative error, highlighting the iterations corresponding to the refinement process. As we refine the domain, the loss function increases, indicating that the challenges in achieving accurate approximations are near the singularity.

In Figure~\ref{fig:ex_ref_errorall}, we include a reference case with two subdomains and $60$ Fourier modes along the long axis, i.e., a total of $3600$ test functions. We highlight that, without increasing the computational cost on a global scale, our refinement strategy results in a relative $\Hoo$-error of $1.20\%$, using three times less test functions than the reference case where the error is recorded as $1.58\%$.

\begin{figure}[ht!]
	\centering
	\begin{subfigure}[b]{0.45\textwidth}
		\centering
		\includegraphics[width=1\textwidth]{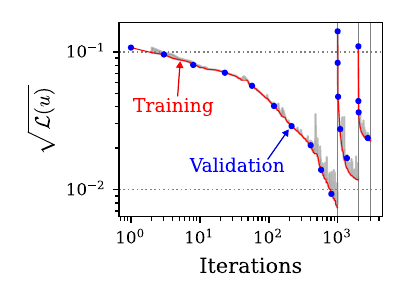}
		\vspace{-1cm}
		\caption{The evolution of the loss function on both the training and validation.}
		\label{fig:ex_ref_lossall}
	\end{subfigure}
	\hspace{1cm}
	\begin{subfigure}[b]{0.45\textwidth}
		\centering
		\includegraphics[width=1\textwidth]{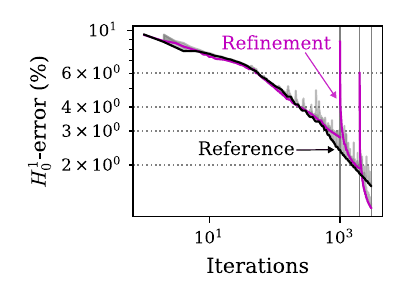}
		\vspace{-1cm}
		\caption{The evolution of the relative  $H_0^1$-error in percentage. }
		\label{fig:ex_ref_errorall}
	\end{subfigure}
	\caption{The loss and the error for the model Case~\hyperref[sec:case3]{3}.}
	\label{fig:ex1_ref_lossanderror}
\end{figure}

\subsection{Case 4. Singular 1D problem}\label{sec:case4}

We now consider the following problem in the weak form: find $u\in H^1_0(0,\pi)$ satisfying
\begin{equation}\label{eq:boundsing}
	\int_0^\pi u'(x)v'(x)-f(x)v(x)\dxx = 0 \qquad \forall v\in H^1_0(0,\pi),
\end{equation}
where $f$ is chosen so that the unique solution is given by 
\begin{equation}
	u^*(x)=x^\alpha(\pi-x),
\end{equation}
and $\alpha=0.7$. We have that $f\not\in L^2(0,\pi)$, but the integral $v\mapsto \int_0^\pi f(x)v(x)\dxx$ corresponds to a continuous linear functional in $H^1_0(0,\pi)$, and thus,~\eqref{eq:boundsing} is a well-defined ordinary differential equation (ODE). As $\alpha<1$, $u'$ is infinite at $0$, and a successful adaptive method refines in a neighborhood of $x=0$. 

Here, we let the initial learning rate to be $10^{-3.5}$. The automatic adaptive refinement process starts with the following subdomain decomposition:
\begin{equation}\label{eq:initdescomp}
	\left\lbrace \Om_i := \left( (i-1)\frac{\pi}{4},(i+1)\frac{\pi}{4}\right)  \right\rbrace_{i=1,2,3}, 
\end{equation}
and we iteratively refine the initial subdomain decomposition five times. We prescribe five Fourier modes for each subdomain (at all levels), and during each iteration of the automatic adaptive Algorithm~\hyperref[Algo1]{1}, we train the NN for a total of $500$ iterations, except for the final step, which employs $1000$ iterations. Moreover, we use a refinement threshold $\tau=0.66$ in Algorithm~\hyperref[Algo1]{1}. Concerning the numerical integration, we generate $500$ geometrically distributed integration points, and, for each refined subdomain, integration is performed using the same quadrature but restricted to that particular subdomain. Here, the relative $\Hoo$-norm of the error is computed using an overkilling quadrature with $5000$ geometrically distributed integration points.

Figure~\ref{fig:SolutionandErrorEx3tot} presents the results from the adaptive refinement process, displaying both the solution and its derivative. Additionally, we illustrate the pointwise errors associated with the final model.
\begin{figure}[ht!]
	\begin{subfigure}[b]{1\textwidth}
		\centering
		\includegraphics{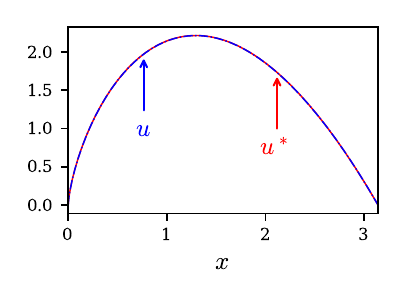}
		\includegraphics{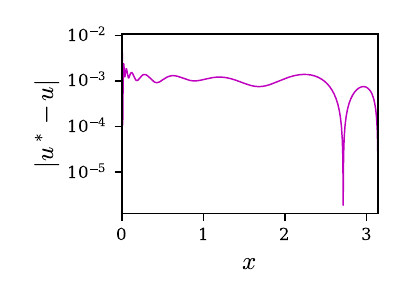}
		\vspace{-0.7cm}
		\caption{Exact and approximate solution (left) and the error in the solution (right).}
		\label{fig:SolutionandErrorEx3}
	\end{subfigure}
	\begin{subfigure}[b]{1\textwidth}
		\centering
		\includegraphics{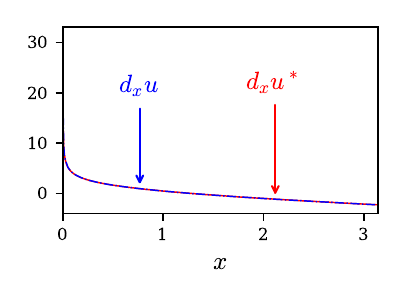}
		\includegraphics{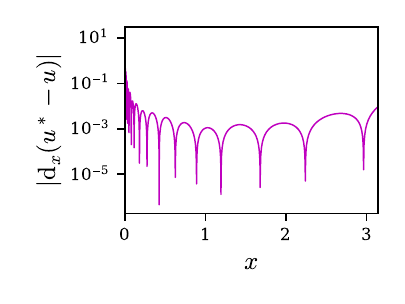}
		\vspace{-0.7cm}
		\caption{Derivative of the exact and approximate solution (left) and the error in the derivative of the solution (right).}
		\label{fig:DerSolutionandErrorEx3}
	\end{subfigure}
	\caption{The solution and the derivative of the solution for the model Case \hyperref[sec:case4]{4}.}
	\label{fig:SolutionandErrorEx3tot}
\end{figure}

The subdomain decomposition obtained at each iteration of the automatic adaptive Algorithm~\hyperref[Algo1]{1} is shown in Figure~\ref{fig:subdex3}. Our method successfully refines the domain $\Om$ near the left boundary, reducing the truncation error solely in that specific region of $\Om$ without introducing additional computational costs in other subdomains. We obtain $8$ subdomains in the last level, implying that we use a total of $40$ localized Fourier modes in the previous iterations.
\begin{figure}[ht!]
	\centering
	\includegraphics{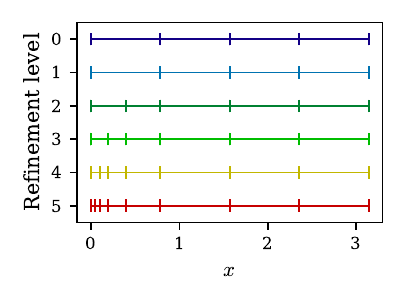}
	\vspace{-0.7cm}
	\caption{The subdomain decomposition at each refinement level for the model Case \hyperref[sec:case4]{4}.}
	\label{fig:subdex3}
\end{figure}

Figure~\ref{fig:ex3_lossanderror} illustrates the results of the model Case \hyperref[sec:case4]{4} in terms of the loss function and the $\Hoo$-norm of the error. There is a consistent relation between the loss function and the error, with both decreasing throughout each refinement step. However, we also observe a deterioration of the equivalence constants as we refine.
\begin{figure}[ht!]
	\centering
	\begin{subfigure}[b]{0.495\textwidth}
		\centering
		\includegraphics{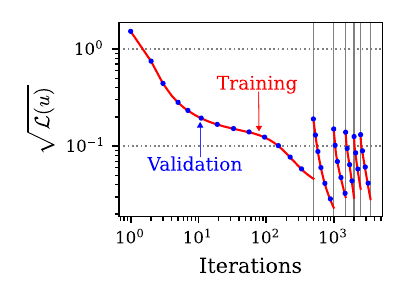}
		\vspace{-0.5cm}
		\caption{The evolution of the loss function during the whole adaptive process and on both the training and validation sets.}
		\label{fig:ex3_lossall}
	\end{subfigure}
	\begin{subfigure}[b]{0.495\textwidth}
		\centering
		\includegraphics{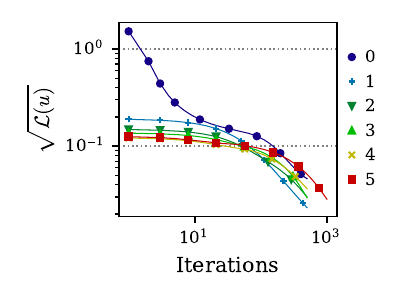}
		\vspace{-0.5cm}
		\caption{The evolution of the loss function for each refinement level. \\~}
		\label{fig:ex3_losseach}
	\end{subfigure}
	\begin{subfigure}[b]{0.495\textwidth}
		\centering
		\includegraphics{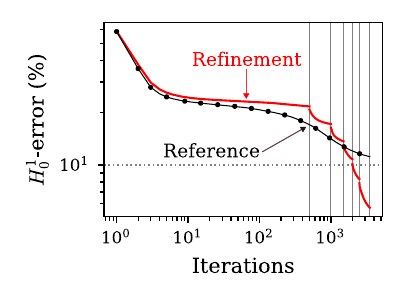}
		\vspace{-0.5cm}
		\caption{The evolution of the relative $H_0^1$-error during the whole adaptive process. }
		\label{fig:ex3_errorall}
	\end{subfigure}
	\begin{subfigure}[b]{0.495\textwidth}
		\centering
		\includegraphics{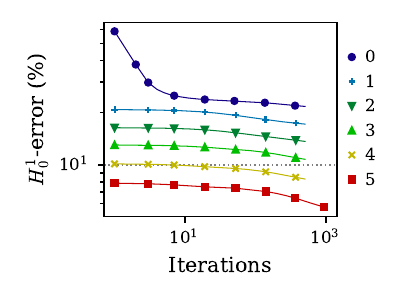}
		\vspace{-0.5cm}
		\caption{The evolution of the relative $H_0^1$-error for each refinement level.}
		\label{fig:ex3_erroreach}
	\end{subfigure}
	\caption{The loss and the error for the model Case \hyperref[sec:case4]{4}.}
	\label{fig:ex3_lossanderror}
\end{figure}

In Figure~\ref{fig:ex3_errorall}, we include the result of a reference model, which employs a non-adaptive approach. In this reference model, no subdomain decomposition of $\Om$ and no refinement process is executed. Moreover, in the reference model, we employ as many Fourier modes ($40$ modes) as the final subdomain decomposition achieved by our algorithm. Those $40$ Fourier modes are evenly distributed throughout the entire domain, and we perform integration using the same set of $500$ geometrically distributed points. This comparison effectively underscores the effectiveness of implementing adaptive strategies to enhance the solution and reduce the total error by mitigating localized errors.

\subsection{Case 5. High-gradients in 1D}\label{sec:case5}

We consider the following ODE in variational form: find $u\in H^1_0(0,\pi)$ satisfying Eq.~\eqref{eq:boundsing}, where $f$ is such that the exact solution is
\begin{equation}
	u^*(x)=x(x-\pi)\exp\left(-120\left(x-\frac{\pi}{2}\right)^2\right).
\end{equation}

The function $u^*$ is smooth but is mostly constant near the boundary and exhibits a prominent peak and a corresponding large derivative near $x=\frac{\pi}{2}$. 

The adaptive refinement process starts with the subdomain decomposition~\eqref{eq:initdescomp} and a learning rate of $10^{-3.5}$. Here, we also prescribe five Fourier modes for each subdomain (at all levels) and iteratively refine the subdomain decomposition five times. We set the refinement constant $\tau$ to $0.66$. During each iteration of the automatic adaptive Algorithm~\hyperref[Algo1]{1}, we train the NN for $500$ iterations. For the numerical integration, we generate $500$ integration points geometrically distributed around $\frac{\pi}{2}$ and on each subdomain, we use the same corresponding points. Moreover, the error is computed using $5000$ geometrically distributed integration points. 

Figure~\ref{fig:SolutionandErrorEx4tot} presents the automatic adaptive refinement process results, showcasing both the solution and its derivative. There, we also plot the corresponding pointwise errors. 
\begin{figure}[ht!]
	\begin{subfigure}[b]{1\textwidth}
		\centering
		\includegraphics{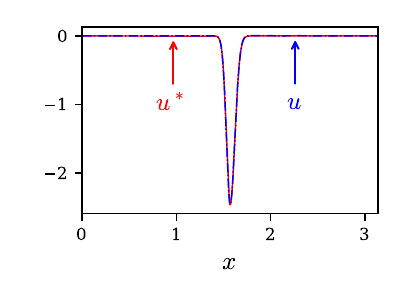}
		\includegraphics{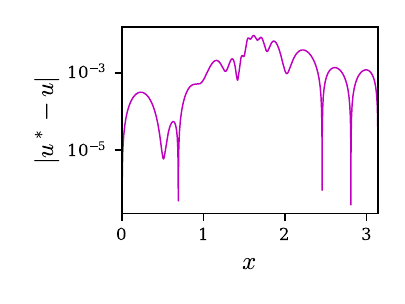}
		\vspace{-0.7cm}
		\caption{Exact and approximate solution (left) and the error in the solution (right).}
		\label{fig:SolutionandErrorEx4}
	\end{subfigure}
	\begin{subfigure}[b]{1\textwidth}
		\centering
		\includegraphics{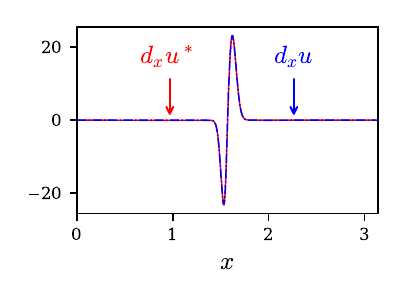}
		\includegraphics{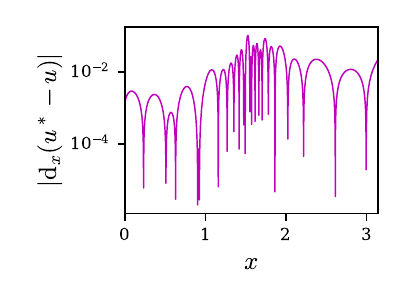}
		\vspace{-0.7cm}
		\caption{Derivative of the exact and approximate solution (left) and the error in the derivative of the solution (right).}
		\label{fig:DerSolutionandErrorEx4}
	\end{subfigure}
	\caption{The solution and the derivative of the solution for the model Case \hyperref[sec:case5]{5}.}
	\label{fig:SolutionandErrorEx4tot}
\end{figure}

Figure~\ref{fig:subdex4} corresponds to the subdomain decomposition obtained at each iteration of the automatic adaptive Algorithm~\hyperref[Algo1]{1}. The refinement process focuses on the region surrounding the high-gradients of the solution near $\frac{\pi}{2}$. In the final iteration, we achieve $12$ subdomains using $60$ localized Fourier modes.
\begin{figure}[ht!]
	\centering
	\includegraphics{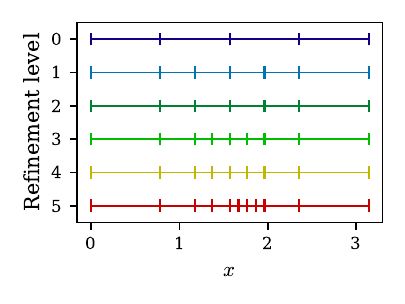}
	\vspace{-0.7cm}
	\caption{The subdomain decomposition at each refinement level for the model Case \hyperref[sec:case5]{5}.}
	\label{fig:subdex4}
\end{figure}

Figure~\ref{fig:ex4_lossanderror} illustrates the evolution of the loss function and the $\Hoo$-norm of the error throughout the entire adaptive process. In Figure~\ref{fig:ex4_errorall}, the reference model utilizes $60$ Fourier modes, the same number of modes employed in our algorithm's final adapted subdomain decomposition. There, the modes are not localized, and the computation of the reference is much more computationally expensive than the solution obtained by our adaptive method. We achieve relative $H_0^1$-errors of $0.00133\%$ by localizing the error. Note that in the initial stages of this example, both the reference and the refined solution display oscillating errors around $100\%$ during the first $1000$ iterations. This indicates that the number of Fourier modes in these levels does not adequately represent the solution and that the improvement in the solution only occurs when a sufficient number of Fourier modes are considered. We include an additional reference curve using $25$ Fourier modes to highlight this behavior. In this case, the error stops decreasing after $2.5\%$, and the truncation error dominates. 

\begin{figure}[ht!]
	\centering
	\begin{subfigure}[b]{0.495\textwidth}
		\centering
		\includegraphics{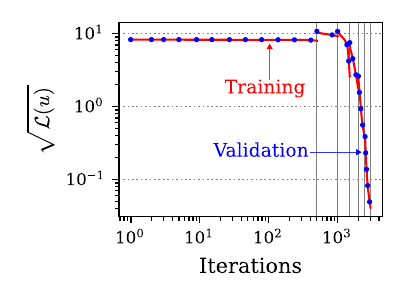}
		\vspace{-0.5cm}
		\caption{The evolution of the loss function during the whole adaptive process and on both the training and validation sets.}
		\label{fig:ex4_lossall}
	\end{subfigure}
	\begin{subfigure}[b]{0.495\textwidth}
		\centering
		\includegraphics{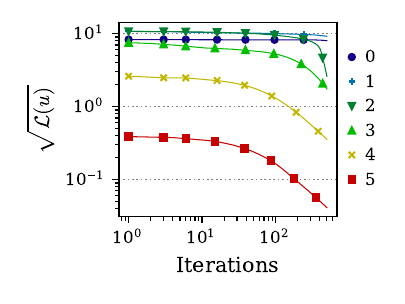}
		\vspace{-0.5cm}
		\caption{The evolution of the loss function for each refinement level. \\~}
		\label{fig:ex4_losseach}
	\end{subfigure}
	\begin{subfigure}[b]{0.495\textwidth}
		\centering
		\includegraphics{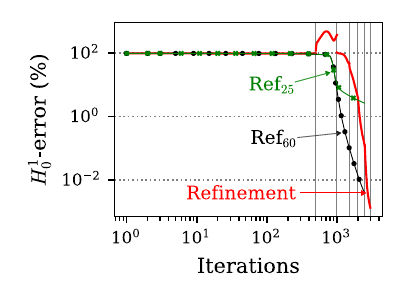}
		\vspace{-0.5cm}
		\caption{The evolution of the relative $H_0^1$-error during the whole adaptive process. }
		\label{fig:ex4_errorall}
	\end{subfigure}
	\begin{subfigure}[b]{0.495\textwidth}
		\centering
		\includegraphics{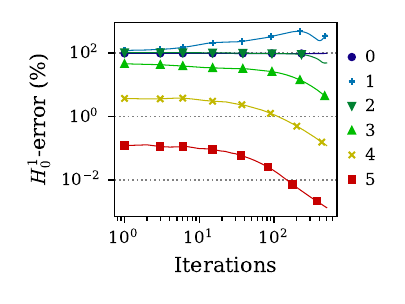}
		\vspace{-0.5cm}
		\caption{The evolution of the relative $H_0^1$-error for each refinement level.}
		\label{fig:ex4_erroreach}
	\end{subfigure}
	\caption{The loss and the error for the model Case \hyperref[sec:case5]{5}.}
	\label{fig:ex4_lossanderror}
\end{figure}


\section{Conclusions}\label{sec:conc}

This work extends the Deep Fourier Residual (DFR) method proposed in~\cite{ taylor2023deep} to handle polygonal geometries and to deal with low-regularity solutions. We propose an iterative subdomain decomposition algorithm guided by an error indicator based on the neural network's loss function. Our method overcomes some of the limitations of traditional Fourier-based approaches. We successfully approximate finer-scale structures in solutions using high-frequency test functions only in specific domain regions. Our adaptive approach eliminates the need for excessive Fourier modes with global support, making it a good fit for PDEs with solutions with peaks and singularities.

Under reasonable assumptions, our method provides good-quality solutions while avoiding the computational cost associated with the inversion of the Gram matrix, a requirement to calculate the loss function in the theory of the Robust Variational Physics-Informed Neural Networks~\cite{ rojas2023robust}. Furthermore, integrating a hybrid optimization approach accelerates Adam's convergence rate, making the training process more efficient. Our numerical experiments consistently validate the strong correlation between the loss function and the $\Hoo$-norm of the error. In 2D, we demonstrate the method's performance using two non-rectangular domains, while in 1D, we provide examples that involve solutions with very low regularity, effectively highlighting the adaptive capabilities of the DFR method. 

We emphasize that through our numerical experiments, we open up a discussion for future work on highly relevant topics, such as the selection of the initial learning rate when using adaptive processes, which may vary depending on the spatial resolution, as well as numerical integration methods that may require a-priori knowledge and on-line modifications when dealing with low-regularity solutions.

\section*{Acknowledgements}

Authors have received funding from the Spanish Ministry of Science and Innovation projects with references TED2021-132783B-I00, PID2019-108111RB-I00 (FEDER/AEI), and PDC2021-121093-I00 (MCIN / AEI / 10.13039 / 501100011033 / Next Generation EU), the ``BCAM Severo Ochoa'' accreditation of excellence CEX2021-001142-S / MICIN / AEI / 10.13039 / 501100011033; the Spanish Ministry of Economic and Digital Transformation with Misiones Project IA4TES (MIA.2021.M04.008 / NextGenerationEU PRTR); and the Basque Government through the BERC 2022-2025 program, the Elkartek project BEREZ-IA (KK-2023 / 00012),, and the Consolidated Research Group MATHMODE (IT1456-22) given by the Department of Education.

 \bibliographystyle{siam}
\bibliography{biblio}

\begin{thebibliography}{10}

\bibitem{babuvvska1978error}
{\sc I.~Babuv{\v{s}}ka and W.~C. Rheinboldt}, {\em {Error estimates for
  adaptive finite element computations}}, {SIAM Journal on Numerical Analysis},
  15 (1978), pp.~736--754.

\bibitem{babuvska11984performance}
{\sc I.~Babuv{\v{s}}ka and M.~Suri}, {\em {The p- and h-p versions of the
  finite element method, an overview}}, {Computer Methods in Applied Mechanics
  and Engineering}, 80 (1990), pp.~5--26.

\bibitem{babuska1981p}
{\sc I.~Babuv{\v{s}}ka, B.~A. Szabo, and I.~N. Katz}, {\em {The p-version of
  the finite element method}}, {SIAM journal on numerical analysis}, 18 (1981),
  pp.~515--545.

\bibitem{badea2022convergence}
{\sc L.~Badea}, {\em {On the Convergence of the Damped Additive Schwarz Methods
  and the Subdomain Coloring}}, {Mathematical and Computational Applications},
  27 (2022), p.~59.

\bibitem{badia2024finite}
{\sc S.~Badia, W.~Li, and A.~F. Mart{\'\i}n}, {\em {Finite element interpolated
  neural networks for solving forward and inverse problems}}, {Computer Methods
  in Applied Mechanics and Engineering}, 418 (2024), p.~116505.

\bibitem{berrone2022solving}
{\sc S.~Berrone, C.~Canuto, and M.~Pintore}, {\em {Solving PDEs by variational
  physics-informed neural networks: an a posteriori error analysis}}, {ANNALI
  DELL'UNIVERSITA'DI FERRARA}, 68 (2022), pp.~575--595.

\bibitem{berrone2022variational}
\leavevmode\vrule height 2pt depth -1.6pt width 23pt, {\em {Variational physics
  informed neural networks: the role of quadratures and test functions}},
  {Journal of Scientific Computing}, 92 (2022), p.~100.

\bibitem{britanak2010discrete}
{\sc V.~Britanak, P.~C. Yip, and K.~R. Rao}, {\em {Discrete cosine and sine
  transforms: general properties, fast algorithms and integer approximations}},
  Elsevier, 2010.

\bibitem{calo2019}
{\sc V.~M. Calo, A.~Ern, I.~Muga, and S.~Rojas}, {\em {An adaptive stabilized
  conforming finite element method via residual minimization on dual
  discontinuous Galerkin norms}}, Computer Methods in Applied Mechanics and
  Engineering, 363 (2020), p.~112891.

\bibitem{chan1994domain}
{\sc T.~F. Chan and T.~P. Mathew}, {\em {Domain decomposition algorithms}},
  {Acta numerica}, 3 (1994), pp.~61--143.

\bibitem{ciarlet2002finite}
{\sc P.~G. Ciarlet}, {\em {The finite element method for elliptic problems}},
  SIAM, 2002.

\bibitem{ciarlet2013linear}
\leavevmode\vrule height 2pt depth -1.6pt width 23pt, {\em {Linear and
  nonlinear functional analysis with applications}}, vol.~130, SIAM, 2013.

\bibitem{cier2021}
{\sc R.~J. Cier, S.~Rojas, and V.~M. Calo}, {\em Automatically adaptive,
  stabilized finite element method via residual minimization for heterogeneous,
  anisotropic advection–diffusion–reaction problems}, Computer Methods in
  Applied Mechanics and Engineering, 385 (2021), p.~114027.

\bibitem{cyr2020robust}
{\sc E.~C. Cyr, M.~A. Gulian, R.~G. Patel, M.~Perego, and N.~A. Trask}, {\em
  {Robust training and initialization of deep neural networks: An adaptive
  basis viewpoint}}, in Mathematical and Scientific Machine Learning, PMLR,
  2020, pp.~512--536.

\bibitem{davies1995spectral}
{\sc E.~B. Davies}, {\em {Spectral Theory and Differential Operators}},
  Cambridge Studies in Advanced Mathematics, Cambridge University Press, 1995.

\bibitem{demkowicz1985h}
{\sc L.~Demkowicz, P.~Devloo, and J.~T. Oden}, {\em {On an h-type
  mesh-refinement strategy based on minimization of interpolation errors}},
  {Computer Methods in Applied Mechanics and Engineering}, 53 (1985),
  pp.~67--89.

\bibitem{dolean2015introduction}
{\sc V.~Dolean, P.~Jolivet, and F.~Nataf}, {\em {An introduction to domain
  decomposition methods: algorithms, theory, and parallel implementation}},
  SIAM, 2015.

\bibitem{weinan2021algorithms}
{\sc W.~E, J.~Han, and A.~Jentzen}, {\em {Algorithms for solving high
  dimensional PDEs: from nonlinear Monte Carlo to machine learning}},
  {Nonlinearity}, 35 (2021), p.~278.

\bibitem{har2011geometric}
{\sc S.~Har-Peled}, {\em {Geometric approximation algorithms}}, no.~173,
  American Mathematical Soc., 2011.

\bibitem{heinlein2021combining}
{\sc A.~Heinlein, A.~Klawonn, M.~Lanser, and J.~Weber}, {\em {Combining machine
  learning and domain decomposition methods for the solution of partial
  differential equations—A review}}, {GAMM-Mitteilungen}, 44 (2021),
  p.~e202100001.

\bibitem{huang2006extreme}
{\sc G.-B. Huang, Q.-Y. Zhu, and C.-K. Siew}, {\em {Extreme learning machine:
  theory and applications}}, {Neurocomputing}, 70 (2006), pp.~489--501.

\bibitem{jagtap2020conservative}
{\sc A.~D. Jagtap, E.~Kharazmi, and G.~E. Karniadakis}, {\em {Conservative
  physics-informed neural networks on discrete domains for conservation laws:
  Applications to forward and inverse problems}}, {Computer Methods in Applied
  Mechanics and Engineering}, 365 (2020), p.~113028.

\bibitem{karniadakis2021physics}
{\sc G.~E. Karniadakis, I.~G. Kevrekidis, L.~Lu, P.~Perdikaris, S.~Wang, and
  L.~Yang}, {\em {Physics-informed machine learning}}, {Nature Reviews
  Physics}, 3 (2021), pp.~422--440.

\bibitem{kharazmi2019vpinns}
{\sc E.~Kharazmi, Z.~Zhang, and G.~Karniadakis}, {\em {VPINNs: variational
  physics-informed neural networks for solving partial differential
  equations}}, {arXiv preprint arXiv:1912.00873},  ({2019}).

\bibitem{kharazmi2021hp}
{\sc E.~Kharazmi, Z.~Zhang, and G.~E. Karniadakis}, {\em {hp-VPINNs:
  Variational physics-informed neural networks with domain decomposition}},
  {Computer Methods in Applied Mechanics and Engineering}, 374 (2021),
  p.~113547.

\bibitem{khodayi2020varnet}
{\sc R.~Khodayi-Mehr and M.~Zavlanos}, {\em {VarNet: Variational neural
  networks for the solution of partial differential equations}}, in Learning
  for Dynamics and Control, PMLR, 2020, pp.~298--307.

\bibitem{khodayi2020deep}
{\sc R.~Khodayi-mehr and M.~M. Zavlanos}, {\em {Deep learning for robotic mass
  transport cloaking}}, {IEEE Transactions on Robotics}, 36 (2020),
  pp.~967--974.

\bibitem{kingma2015ADAM}
{\sc D.~P. Kingma and J.~Ba}, {\em {Adam: A method for stochastic
  optimization}}, {arXiv preprint arXiv:1412.6980},  (2014).

\bibitem{lagaris1998artificial}
{\sc I.~E. Lagaris, A.~Likas, and D.~I. Fotiadis}, {\em {Artificial neural
  networks for solving ordinary and partial differential equations}}, {IEEE
  transactions on neural networks}, 9 (1998), pp.~987--1000.

\bibitem{magueresse2023adaptive}
{\sc A.~Magueresse and S.~Badia}, {\em {Adaptive quadratures for nonlinear
  approximation of low-dimensional PDEs using smooth neural networks}}, {arXiv
  preprint arXiv:2303.11617},  (2023).

\bibitem{mehta2004handbook}
{\sc D.~P. Mehta and S.~Sahni}, {\em {Handbook of data structures and
  applications}}, Chapman and Hall/CRC, 2004.

\bibitem{moseley2023finite}
{\sc B.~Moseley, A.~Markham, and T.~Nissen-Meyer}, {\em {Finite Basis
  Physics-Informed Neural Networks (FBPINNs): a scalable domain decomposition
  approach for solving differential equations}}, {Advances in Computational
  Mathematics}, 49 (2023), p.~62.

\bibitem{nevcas1962methode}
{\sc J.~Ne{\v{c}}as}, {\em {Sur une m{\'e}thode pour r{\'e}soudre les
  {\'e}quations aux d{\'e}riv{\'e}es partielles du type elliptique, voisine de
  la variationnelle}}, {Annali della Scuola Normale Superiore di Pisa-Scienze
  Fisiche e Matematiche}, 16 (1962), pp.~305--326.

\bibitem{oden2017applied}
{\sc J.~T. Oden and L.~Demkowicz}, {\em {Applied functional analysis}}, CRC
  press, 2017.

\bibitem{paszynski2021deep}
{\sc M.~Paszy{\'n}ski, R.~Grzeszczuk, D.~Pardo, and L.~Demkowicz}, {\em Deep
  learning driven self-adaptive hp finite element method}, in International
  Conference on Computational Science, Springer, 2021, pp.~114--121.

\bibitem{raissi2019physics}
{\sc M.~Raissi, P.~Perdikaris, and G.~E. Karniadakis}, {\em {Physics-informed
  neural networks: A deep learning framework for solving forward and inverse
  problems involving nonlinear partial differential equations}}, {Journal of
  Computational Physics}, {378} ({2019}), pp.~{686--707}.

\bibitem{rivera2022quadrature}
{\sc J.~A. Rivera, J.~M. Taylor, {\'A}.~J. Omella, and D.~Pardo}, {\em {On
  quadrature rules for solving partial differential equations using neural
  networks}}, {Computer Methods in Applied Mechanics and Engineering}, 393
  (2022), p.~114710.

\bibitem{rojas2023robust}
{\sc S.~Rojas, P.~Maczuga, J.~Muñoz-Matute, D.~Pardo, and M.~Paszynski}, {\em
  {Robust Variational Physics-Informed Neural Networks}}, 2023.

\bibitem{rojas2021}
{\sc S.~Rojas, D.~Pardo, P.~Behnoudfar, and V.~M. Calo}, {\em Goal-oriented
  adaptivity for a conforming residual minimization method in a dual
  discontinuous galerkin norm}, Computer Methods in Applied Mechanics and
  Engineering, 377 (2021), p.~113686.

\bibitem{shukla2021parallel}
{\sc K.~Shukla, A.~D. Jagtap, and G.~E. Karniadakis}, {\em {Parallel
  physics-informed neural networks via domain decomposition}}, {Journal of
  Computational Physics}, 447 (2021), p.~110683.

\bibitem{sirignano2018dgm}
{\sc J.~Sirignano and K.~Spiliopoulos}, {\em {DGM: A deep learning algorithm
  for solving partial differential equations}}, {Journal of computational
  physics}, 375 (2018), pp.~1339--1364.

\bibitem{Manuela2023deep}
{\sc J.~M. Taylor, M.~Bastidas, D.~Pardo, and I.~Muga}, {\em Deep fourier
  residual method for solving time-harmonic maxwell's equations}, arXiv
  preprint arXiv:2305.09578,  (2023).

\bibitem{taylor2023deep}
{\sc J.~M. Taylor, D.~Pardo, and I.~Muga}, {\em {A Deep Fourier Residual method
  for solving PDEs using Neural Networks}}, {Computer Methods in Applied
  Mechanics and Engineering}, {405} ({2023}), p.~{115850}.

\bibitem{toselli2004domain}
{\sc A.~Toselli and O.~Widlund}, {\em {Domain decomposition methods-algorithms
  and theory}}, vol.~34, Springer Science \& Business Media, 2004.

\bibitem{wu2023comprehensive}
{\sc C.~Wu, M.~Zhu, Q.~Tan, Y.~Kartha, and L.~Lu}, {\em {A comprehensive study
  of non-adaptive and residual-based adaptive sampling for physics-informed
  neural networks}}, {Computer Methods in Applied Mechanics and Engineering},
  403 (2023), p.~115671.

\end{thebibliography}

\appendix
\addcontentsline{toc}{section}{Appendices}
\renewcommand{\thesubsection}{\Alph{section}.\Roman{subsection}}

\section{Domain decomposition over polygonal domains}\label{app:2}
Here we provide rigorous results to apply the DFR-based domain decomposition method to rectilinear polygons, in particular to the L-shape and pentagonal domains. Our aim is to show that Assumption~\ref{as:a2} holds in the given domains, to ensure the validity of our method.

For simplicity, here we focus only on decompositions of functions in $\Hoo$. An extension of our method to the case of Neumann (and Robin) boundary conditions is straightforward.

Consider $\{\Om_i\}_{i=1}^m$ to be a collection of rectangular open domains, not necessarily parallel. We define $\Om=\bigcup\limits_{i=1}^m \Om_i$.  Our approach is to use partitions of unity. If a Lipschitz partition of unity over $\{\Omega_i\}_{i=1}^m$, $\{\rho_i\}_{i=1}^m$ exists, then the mapping $v\mapsto \rho_i v$ would give a decomposition of $v=\sum\limits_{i=1}^mv_i$ with $v_i\in H^1_0(\Om_i)$. Then:
\begin{equation}
	\int_\Om |\nabla (\rho_i v)|^2\dx \leq 2\int_\Om v^2|\nabla \rho_i|^2 + \rho_i^2 |\nabla v|^2\dx\leq  2 \max(\|\nabla\rho_i\|_{\infty},1)\int_\Om |\nabla v|^2+v^2\dx. 
\end{equation} 

Unfortunately, some geometries do not permit Lipschitz partitions of unity. If the Lipschitz constant does not grow too rapidly, as we are considering functions in $\Hoo$, we may employ Hardy's inequality~\cite{nevcas1962methode}, which states that there exists a constant $C>0$ such that 
\begin{equation}
	\int_\Om \frac{1}{d(\tx,\partial\Om)^2}v^2\dx \leq C \int_\Om |\nabla v|^2\dx
\end{equation}
for all $v\in \Hoo$. Thus, if $|\nabla\rho_i|^2\leq C \frac{1}{d(\tx,\partial\Om)^2}$, we may define the decomposition $v=\sum\limits_{i=1}^m\rho_iv$, despite $\{\rho_i\}_{i=1}^m$ having unbounded gradient at the boundary. 

Here, the methodology is as follows:
\begin{enumerate}
	\item Define an appropriate continuous partition of unity over the subdomains (Definition~\ref{defPartitionUnity}).
	\item Demonstrate that the partition is Lipschitz on closed subsets of $\bar{\Om}$ that do not contain certain singular points (Corollary~\ref{corLipRegular}). 
	\item Construct explicit local partitions of unity in neighbourhoods of the problematic points that, while not Lipschitz, have controlled growth of the gradient that permits $\HoD$-estimates (Propositions~\ref{propDecompL} and~\ref{propDecomp5}). 
	\item Combine the global partition of unity from (2) with the local partitions from (3) via cutoff-functions to provide the required decomposition (Theorem~\ref{thm:Lshapepent_app}).
\end{enumerate}
Our approach to the first two steps is general. The local analysis of singular points will depend on the particular geometry, and we do not aim to provide a general theorem here, rather an outline on how the method may be applied to particular cases in a replicable way. The local partitions of unity will be based on a simple, explicit representation in polar coordinates.

\begin{definition}\label{defPartitionUnity}
	Let $\Om=\bigcup\limits_{i=1}^m\Om_i$ be a Lipschitz polygon. We define the functions $p_i:\bar{\Om}\to\mathbb{R}$ and ${\rho_i:\bar{\Om}\setminus\{x:\sum\limits_{i=1}^mp_i(\tx)=0\}\to\mathbb{R}}$ by 
	\begin{equation}
			p_i(\tx)= d(\tx,\Om\setminus \Om_i)=\inf\limits_{u\in\Om\setminus \Om_i} |x-y|, \quad \text{ and } \quad
			\rho_i(\tx)=\frac{p_i(\tx)}{\sum\limits_{j=1}^np_j(\tx)}.
	\end{equation}
	 We denote the partition by $\bm\rho=\{\rho_i\}_{i=1}^m$. 
\end{definition}

\begin{proposition}\label{prop:PartitionCont}
	$\bm\rho$ defines a continuous partition of unity of $\Om=\bigcup\limits_{i=1}^m\Om_i$.
\end{proposition}
\begin{proof}
	As each $p_i$ is a distance function, it is necessarily continuous and non-negative. If $\tx\in \Om_i$ for some $i$, as $\Om_i$ is open, $p_i(\tx)$ is necessarily positive. In particular, the functions $\{\rho_i\}_{i=1}^m$ are well-defined on $\Om$ as $\sum\limits_{i=1}^mp_i(\tx)\neq 0$ on $\Om$. $p_i(\tx)$ is trivially zero for $\tx\in\Om\setminus \Om_i$. Thus it is immediate that $\rho_i$ is positive on $\Om_i$ and zero otherwise, and the construction gives $\sum\limits_{i=1}^m\rho_i(\tx)=1$. 
\end{proof}

Whilst $\bm\rho$ is continuous on $\Omega$, we need to bound its gradient. Furthermore, we note that its definition is implicitly ill-defined at regions in $\partial\Omega$ where $\sum\limits_{i=1}^mp_i(\tx)=0$. We introduce the following definition. 

\begin{definition}
	Let $\Om=\bigcup\limits_{i=1}^m \Om_i$ be a Lipschitz domain. We define a point $\tx\in\bar{\Om}$ to be a {\it $\bm\rho$-singular point} if $\sum\limits_{i=1}^m p_i(\tx)=0$, and to be a {\it $\bm\rho$-regular point} otherwise.
\end{definition}

We choose the name {\it $\bm\rho$-singular point} as we will show later that these are precisely the points at which the derivative of the partition of unity we consider is unbounded, it is clear from the definition of $\bm\rho$ that these points may introduce singularities. This is made precise in the following proposition

\begin{proposition}
	Each $\rho_i$ is differentiable almost everywhere, and 
	\begin{equation}\label{eqLipEstimate}
		|\nabla \rho_i(\tx)|\leq \frac{n+1}{\sum\limits_{j=1}^np_j(\tx)}
	\end{equation}
	almost everywhere on $\bar{\Om}$. In particular, each $\rho_i$ is Lipschitz on closed subsets of $\bar{\Om}$ that do not contain any singular points.
\end{proposition}
\begin{proof}
	From the argument of Proposition~\ref{prop:PartitionCont}, we have that $\sum\limits_{i=1}^m p_i(\tx)>0$ on $\Omega$. In particular, the set of singular points is of measure zero, as it is contained in $\partial\Omega$. 
	Each $p_i$ is a distance function, and is thus Lipschitz with Lipschitz constant at most $1$ and differentiable almost everywhere. It is then immediate that $\rho_i$ differentiable almost everywhere on $\Omega$ as each $p_j$ is, and it remains only to estimate the derivative. This is straightforward, as 
	\begin{equation}
			\nabla\rho_i(\tx)= \frac{1}{\sum\limits_{j=1}^np_j(\tx)}\nabla p_i(\tx)-p_i(\tx)\frac{\sum\limits_{j=1}^m \nabla p_j(\tx)}{\left(\sum\limits_{j=1}^n p_j(\tx)\right)^2}.
	\end{equation}
	We then use that each $p_j$ is positive giving $\frac{p_i(\tx)}{\sum\limits_{j=1}^n p_j(\tx)}\leq 1$, and that $|\nabla p_i(\tx)|\leq 1$ almost everywhere to obtain 
	\begin{equation}
			|\nabla \rho_i(\tx)|\leq \frac{1}{\sum\limits_{j=1}^np_j(\tx)}+\frac{p_i(\tx)}{\sum\limits_{j=1}^np_j(\tx)}\frac{\sum\limits_{j=1}^n|\nabla\rho_j(\tx)|}{\sum\limits_{j=1}^np_j(\tx)}\leq \frac{1+n}{\sum\limits_{j=1}^np_j(\tx)}. 
	\end{equation}
	In particular, as each $p_i$ is non-negative and continuous, we have that on any closed subset of $\bar{\Om}$ not containing any $\bm\rho$-singular points, $\sum\limits_{i=1}^m p_i(\tx)$ is bounded away from zero and thus the gradient of $\rho_i$ is bounded almost everywhere, giving that $\rho_i$ is Lipschitz on closed sets containing no $\bm\rho$-singular points. 
\end{proof}

Whilst the definition of $\bm\rho$-singular points is useful for estimating the regularity of the partition, it is not immediately clear how they may be identified in practice. Our next proposition gives an equivalent condition based on the geometry that will be more useful for identifying $\bm\rho$-singular points. In particular, this will be used to show that the set of $\bm\rho$-singular points is finite and can be be identified by inspection of the vertices of $\Omega$ and $\Omega_i$, and a local analysis at each is then straightforward.

\begin{proposition}\label{prop:EquivSingular}
	A point $\tx\in \bar{\Om}$ is a $\bm\rho$-regular point if and only if there exists some $\delta>0$ and $i=1,...,m$ such that $B_\delta(\tx)\cap\Omega = B_\delta(\tx)\cap\Omega_i$. 
\end{proposition}
\begin{proof}
	First, assume that $x$ is a $\bm\rho$-regular point. This implies that $\sum\limits_{i=1}^mp_i(\tx)>0$, and in particular $p_i(\tx)>0$ for some $i$. Thus, as $d(\tx,\Omega\setminus\Omega_i)>0$, for sufficiently small $\delta$, $\emptyset = B_\delta(\tx)\cap (\Omega\setminus\Omega_i)=(B_\delta(\tx)\cap\Omega\setminus (B_\delta(\tx)\cap \Omega_i)$. As $\Omega_i\subset\Omega$, this yields $B_\delta(\tx)\cap\Omega=B_\delta(\tx)\cap\Omega_i$.
	
	For the reverse implication, if $B_\delta(\tx)\cap\Omega= B_\delta(\tx)\cap\Omega_i$ for some $i=1,...,m$ and $\delta>0$, then, by reversing the previous argument, $\emptyset=B_\delta(\tx)\setminus(\Omega\setminus\Omega_i)$, so that $B_\delta(\tx)\subset \Omega\setminus\Omega_i$ and thus $p_i(\tx)=d(\tx,\Omega\setminus\Omega_i)>\delta$. 
\end{proof}

With this equivalent characterisation of $\bm\rho$-singular and $\bm\rho$-regular points, we can demonstrate that any point that simultaneously belongs to an edge-segment of $\partial\Omega$ and $\partial\Omega_i$ for some $i$ is a $\bm\rho$-regular point. 

For clarity, we refer to a region $\Gamma\subset \partial U$, on a Lipschitz polygon $U$, as an edge segment if it is a connected region of the boundary that contains no vertices of $\partial U$. In particular, for small $\delta>0$, if $\tx_0$ is in an edge segment of $U$, then, for some $\tv\in\mathbb{R}^2$ and $c\in\mathbb{R}$, $B_\delta(\tx_0)\cap U=\{\tx\in B_\delta (\tx_0): \tx\cdot \tv<c\}$

\begin{proposition}
	If $\Gamma\subset\partial\Om$ is an edge-segment of both $\Om$ and $\Om_i$ for some $i$, then each $\tx\in \Gamma$ is a regular point. 
\end{proposition}
\begin{proof}
	Let $\tx_0\in \Gamma$. As $\Gamma$ is an edge segment for $\Om_i$ and $\Om$, we have that there exists $\delta>0$, $\tv\in\mathbb{R}^2$ and $c\in\mathbb{R}$ with $B_\delta(\tx_0)\cap \Om=\{\tx\in B_\delta (\tx_0): \tx\cdot \tv<c\}$. It must also hold that, for sufficiently small $\delta$, $B_\delta(\tx_0)\cap \Om=\{\tx\in B_\delta (\tx_0): \tx\cdot \tv<c\}$ with the same $v$ and $c$, else $\Om_i$ would not be a subset of $\Om$. In particular, $B_\delta(\tx_0)\cap\Om= B_\delta(\tx_0)\cap \Om_i$, and thus $\tx_0$ is a regular point. 
\end{proof}

\begin{corollary}
	If $\Om=\bigcup\limits_{i=1}^m \Om_i$ is a Lipschitz polygon, then all $\bm\rho$-singular points of $\Om$ are either vertices of $\Om_i$ or vertices of $\Om_i$. In particular, they are finite in number. 
\end{corollary}
\begin{proof}
	We have that all $\bm\rho$-singular points must be in $\partial\Om$ from Proposition~\ref{prop:PartitionCont}. From the previous proposition, any point which is in an edge segment of both $\Om$ and $\Om_i$ is a regular point. Thus, all $\bm\rho$-singular points must correspond either to vertices of $\Om$ or vertices of $\Om_i$, which is a finite number. 
\end{proof}

For a given set of rectangular domains $\{\Om_i\}_{i=1}^m$, it will generally be straightforward to identify their $\bm\rho$-singular points. One need only consider the vertices of each $\Om_i$ and the vertices of $\Om$, and the property obtained in Proposition~\ref{prop:EquivSingular} can be readily verified at each. We illustrate the $\bm\rho$-singular points of the L-shape and pentagonal domains below, along with their local environments. 

\begin{figure}[h!]
	\centering
		\begin{subfigure}{0.45\textwidth}
			\centering
			\begin{tikzpicture}[scale = 1.3]
				\draw[thick,fill=gray!8] (-1,0) rectangle (1,1);
				\draw[thick,fill=gray!8] (0,-1) rectangle (1,1);
				\filldraw (0,0) circle (0.1);
			\end{tikzpicture}\hspace{0.8cm}
			\begin{tikzpicture}[scale=1.3]
				\draw[thick,domain = -90:180] plot({cos(\x)},{sin(\x)});
				\draw[thick,domain =0:90] plot ({0.35*cos(\x)},{0.35*sin(\x)});
				\node[] at (0.5,0.15) {$\frac{\pi}{4}$};
				\draw[thick,green] (1,0)--(0,0)--(-1,0);
				\draw[thick,blue] (0,1)--(0,0)--(0,-1);
				\filldraw[black] (0,0) circle (0.1);
			\end{tikzpicture}
			\caption{The L-shape, the $\bm\rho$-singular point highlighted (left) and the local environment $B_\delta(\tx_0)\cap \Om$ (right).}
		\end{subfigure}
			\begin{subfigure}{0.45\textwidth}
			\centering
			\begin{tikzpicture}[scale = 1.2]
				\draw[thick,fill=gray!8](-1,-1) rectangle (1,1);
				\draw[thick,fill=gray!8] (1,-1)--(2,0)--(1,1)--(0,0)--(1,-1);
				\filldraw[black](1,1) circle (0.1);
				\filldraw[black](1,-1) circle (0.1);
			\end{tikzpicture}
			\begin{tikzpicture}[scale=1.9]
				\draw[thick,domain = 0:135] plot({cos(\x)},{sin(\x)});
				\draw[thick,domain =0:45] plot ({0.35*cos(\x)},{0.35*sin(\x)});
				\node[] at (0.5,0.15) {$\frac{\pi}{4}$};
				\draw[thick,green] (0,1)--(0,0)--(1,0);
				\draw[thick,blue] (0.707,0.707)--(0,0)--(-0.707,0.707);
				\filldraw[black] (0,0) circle (0.1);
			\end{tikzpicture}
			\caption{The pentagonal domain, the $\bm\rho$-singular points (left) and the local environment $B_\delta(\tx_0)\cap \Om$ (right).}
		\end{subfigure}
	\caption{Illustration of the L-shape and pentagonal domains with their $\bm\rho$-singular points highlighted and their corresponding local environments.}
\end{figure}
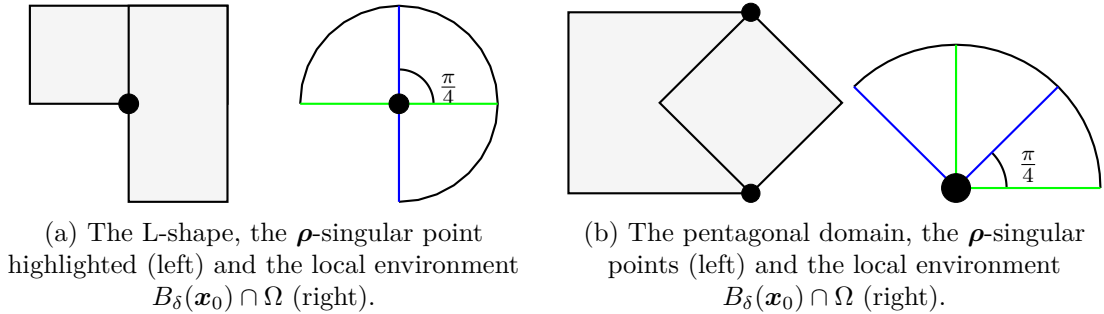

\begin{corollary}\label{corLipRegular}
	Let $\Om=\bigcup\limits_{i=1}^m \Om_i$ be a Lipschitz polygon, and denote its $\bm\rho$-singular points $\{\tx_j\}_{j=1}^n$. Then, for every $\delta>0$, each $\rho_i$ is Lipschitz on the set $\bar{\Om}\setminus \bigcup\limits_{j=1}^n B_\delta(\tx_j)$. 
\end{corollary}
\begin{proof}
	This directly follows from Proposition~\ref{eqLipEstimate}.
\end{proof}

The previous results suggest that the partition of unity $\bm\rho$ can be used away from the finite collection of $\bm\rho$-singular points of $\Om$. Near the $\bm\rho$-singular points, we perform a local decomposition by considering the wedge-shaped segments in polar coordinates. We proceed only focusing only on the particular $\bm\rho$-singular points in the L-shape and pentagonal domains, where similar constructions can easily be envisaged for more general vertices. We do not aim to classify all potential types of $\bm\rho$-singular points nor provide a general proof. 

\begin{proposition}\label{propDecompL}
	Let $\tx_0$ denote the $\bm\rho$-singular point of the L-shape domain, and let $\delta>0$ be small. Then, for $v\in H^1_0(\Om\cap B_\delta(\tx_0))$, there exists $v_i\in H^1_0(\Om_i\cap B_\delta(\tx_0))$ so that $v=v_1+v_2$
\end{proposition}
\begin{proof}
	
	We take an appropriate polar coordinate system centred at $\tx_0$ such that $B_\delta(\tx_0)\cap \Om_1$ corresponds to coordinate space $(r,\theta)\in (0,\delta)\times\left(-\frac{\pi}{2},\frac{\pi}{2}\right)$, $B_\delta(\tx_0)\cap \Om_2$ corresponds to $(0,\delta)\times\left(0,\pi\right)$. We define the function $\tilde{\rho}_1$ in such a coordinate system as 
	\begin{equation}
		\tilde{\rho}_1(r,\theta)=\left\{\begin{array}{l l}
			0 & 0\leq \theta \leq 0\\
			\frac{2\theta}{\pi} & 0<\theta<\frac{\pi}{2}\\
			1 & \theta\geq\frac{\pi}{2}
		\end{array}\right.
	\end{equation}
	$\tilde{\rho}_2$ is then defined to be $1-\tilde{\rho}_1$. It is immediate that $(\tilde{\rho}_1,\tilde{\rho}_2)$ is a continuous partition of unity of $\left(B_\delta(\tx_0)\cap \Om_1)\cup(B_\delta(\tx_0)\cap \Om_2\right)$. Each is differentiable almost everywhere on this set, and its gradient is calculated to satisfy $|\nabla \tilde{\rho}_i|^2=\frac{4}{\pi^2r^2}$ for $\theta\in \left(0,\frac{\pi}{2}\right)$ and zero otherwise. As $\tx_0\in\partial\Om$ and $r=|\tx-\tx_0|$, we then have that $r\geq d(\tx,\partial\Om)$, so 
	\begin{equation}
		|\nabla\tilde{\rho}_i|^2\leq \frac{4}{\pi^2|\tx-\tx_0|^2}\lesssim \frac{1}{d(\tx,\partial\Om)^2}
	\end{equation}
	
	Now, taking $v\in H^1_0(\Om\cap B_\delta(\tx_0))$, we have that 
	\begin{equation}
		\begin{split}
			\int_{\Om\cap B_\delta(\tx_0)}|\nabla (\rho_iv)|^2\dx \leq &2\int_{\Om\cap B_\delta(\tx_0)} \rho_i^2|\nabla v|^2+v^2|\nabla\rho_i|^2\dx \\
			\lesssim & \int_{\Om\cap B_\delta(\tx_0)} |\nabla v|^2+\frac{1}{d(\tx,\partial\Om)^2} v^2\dx
			\lesssim \int_{\Om\cap B_\delta(\tx_0)}|\nabla v|^2=\|v\|_{H^1_0}^2,
		\end{split}
	\end{equation}
	using that $|\rho_i|\leq 1$, the estimate obtained before, and Hardy's inequality in the last line. Thus we  have that $v_i=\tilde{\rho}_i v$ gives the required decomposition, with $v_i\in H^1_0(B_\delta(\tx_0)\cap \Om_i)$. 
\end{proof}

\begin{proposition}\label{propDecomp5}
	Let $\tx_0$ denote a $\bm\rho$-singular point of the pentagonal domain, and let $\delta>0$ be small. Then, for each $v\in H^1_0(\Om\cap B_\delta(\tx_0))$, there exists $v_i\in H^1_0(\Om_i\cap B_\delta(\tx_0))$ so that $v=v_1+v_2$
\end{proposition}
\begin{proof}
	The construction is near identical to the case in the L-shape domain. The local environments of each $\bm\rho$-singular point are similar, and so by constructing the decomposition for one, we obtain the decomposition by a trivial transformation.
	
	We consider the point which has local environment given in polar coordinates by $(r,\theta)\in (0,\delta)\times \left(\frac{\pi}{4},\pi\right)$, with $B_\delta(\tx_0)\cap \Om_1$ corresponding to $(0,\delta)\times \left(\frac{\pi}{2},\pi\right)$ and $B_\delta(\tx_0)\cap \Om_1$ corresponding to $(0,\delta)\times \left(\frac{\pi}{4},\frac{3\pi}{4}\right)$. We define our partition of unity $\{\tilde{\rho}_i\}_{i=1}^2$ by 
	\begin{equation}
		\tilde{\rho}_1(r,\theta)=\left\{\begin{array}{l l}
			1 & \frac{\pi}{4}\leq\theta\leq\frac{\pi}{4},\\
			\frac{3}{2}-\frac{2 \theta}{\pi} & \frac{\pi}{4}<\theta<\frac{3\pi}{4},\\
			0 & \theta\geq\frac{3\pi}{4}.
		\end{array}\right.
	\end{equation}
	Here, $\tilde{\rho}_2$ is defined to be $1-\tilde{\rho}_1$. By a similar argument as before, we have that $|\nabla \tilde{\rho}_i|^2\leq \frac{C}{r^2}\leq \frac{C}{d(\tx,\partial\Om)^2}$, and so by employing Hardy's inequality we have that $v\mapsto \tilde{\rho}_iv=v_i$ defines a partition so that $v=v_1+v_2$ with $v_i \in H^1_0(\Om_i\cap B_\delta(\tx_0))$. 
\end{proof}

\begin{theorem}\label{thm:Lshapepent_app} 
	Let $\Om=\Om_1\cup \Om_2$ be the L-shape or pentagonal domain. Then, for every $v\in \Hoo$, there exists $v_i\in H^1_0(\Om_i)$ such that $v=v_1+v_2$. 
\end{theorem}
\begin{proof}
	Let $\delta\ll 1$ and take $\varphi$ to be a smooth cut-off function that satisfies $\varphi(\tx)=1$ for $|x|\leq \frac{\delta}{2}$ and $\varphi(\tx)=0$ for $|x|\geq \delta$. Let $\{x_j\}_{j=1}^k$ denote the $\bm\rho$-singular points of $\Om$, with $k=1$ in the case of the L-shape domain, $k=2$ in the pentagonal domain. Write $v=v_0+\sum\limits_{j=1}^kv_0^j$, where $v_0(\tx)=\left(1-\sum\limits_{j=1}^k\varphi(\tx-x_j)\right)v(\tx)$, and $v_0^j(\tx) =\varphi(\tx-x_j)v(\tx)$. We then have that $v_0^j\in H^1_0(\Om\cap B_\delta(\tx_j))$, and $v_0\in H^1_0(\Om\setminus\bigcup\limits_{j=1}^kB_\delta(\tx_j))$. By Proposition~\ref{propDecompL} in the case of the L-shape domain and Proposition~\ref{propDecomp5} in the case of the pentagonal domain, we may write $v_0^j = v_{0,1}^j+v_{0,2}^j$ with $v_{0,i}^j\in H^1_0(\Om_i\cap B_\delta(\tx_j))$. By Proposition~\ref{propDecompL}, the partition of unity $(\rho_1,\rho_2)$ is Lipschitz on the support of $v_0$, so $\rho_i v_0\in H^1_0(\Om_i)$. Thus we may define $v_i=\rho_iv_0+\sum\limits_{j=1}^k v_{0,i}^j$, with $v_i\in H^1_0(\Om_i)$ and $v_1+v_2=v$. 
\end{proof}

\begin{remark}[Extension to $\HoD$]
	An extension of our method is straightforward in the case of $\GD \neq \partial \Om$, under the assumption that, for some $\delta>0$,
	 $\partial\Om\cap\left(\bigcup\limits_{j=1}^kB_\delta(\tx_j)\right)\subset\GD$. That is, the $\bm\rho$-singular points are separated from $\partial\Om\setminus\GD$. This is a consequence of Proposition~\ref{propDecompL}, as our partition of unity is Lipschitz up to the boundary outside of the $\bm\rho$-singular points. We note that such decomposition may be impossible in some cases, such as the L-shape domain with $\GD=\emptyset$ and the constant function $v(\tx)=1$. 
\end{remark}

\begin{remark}[Extension to general Lipschitz polygons]
	To extend this procedure to general Lipschitz polygons, it is clear that the only necessity is an equivalent local decomposition to Proposition~\ref{propDecompL} and Proposition~\ref{propDecomp5} at each $\bm\rho$-singular point of $\Om$. The $\bm\rho$-singular points may be readily identified for a given geometry, and a similar decomposition in polar coordinates can be easily imagined for a wide variety of vertex types. As the partition of unity $\{\rho_i\}_{i=1}^m$ given in Definition~\ref{defPartitionUnity} is Lipschitz away from the $\bm\rho$-singular points (c.f. Proposition~\ref{propDecompL}), the same decomposition via local cut-off functions can be employed. 
\end{remark}

\section{Bounds over the $\star$-norm}\label{app:1}
Given a Hilbert space $H$ and a set of closed subspaces $ \{H_i\}_{i=1}^m$ within $H$, we define the minimal energy decomposition of $v \in H$ as
\begin{equation}\label{def.opt.decomp}
	\{P_iv\}_{i=1}^m := \mathrm{arg}\min\left\{\sum\limits_{i=1}^m \|v_i\|_{H}^2:v_i\in H_i \text{ and } v = \sum\limits_{i=1}^m v_i\right\}. 
\end{equation}

To prove Proposition~\ref{prop:star}, we refer to the following lemma, which states the well-posedness of the mappings $\{P_i\}_{i=1}^m$.

\begin{lemma}\label{prop:maps}
	Let $H$ be a Hilbert space. If there exists a set of closed subspaces $\{H_i\}_{i=1}^m$ satisfying Assumption~\ref{as:a2},  then the mappings $\{P_i\}_{i=1}^m$ resulting from the minimal energy decomposition of $v \in H$ are bounded linear operators from $H\to H_i$. Furthermore, 
	\begin{equation}
		\xi \left( \{H_i\}_{i=1}^m \right):=\sup\limits_{v\in H}\frac{1}{\|v\|_{H}}\min\left\{\sqrt{\sum\limits_{i=1}^m \|v_i\|_{H}^2}:v_i\in H_i \text{ and } v = \sum\limits_{i=1}^m v_i \right\}<+\infty. 
	\end{equation}
\end{lemma}

\begin{proof}
	To prove this lemma, we first show that the mapping $(v_i)_{i=1}^m\to \sum\limits_{i=1}^m v_i$ is a closed mapping from a particular Hilbert space to $H$. We can then show that it has a bounded, linear inverse on the orthogonal complement of its kernel via the bounded inverse theorem, which we show to be the minimal energy decomposition. It then results that $\xi \left( \{H_i\}_{i=1}^m \right)$ is the norm of this mapping, which is necessarily finite.

	We start by defining the Hilbert space $\hat{H}=\prod\limits_{i=1}^mH_i$, with the natural inner product \[\langle (v_i^1)_{i=1}^m,(v_i^2)_{i=1}^m \rangle_{\hat{H}}= \sum\limits_{i=1}^m\langle v^1_i,v_i^2\rangle_H.\]
	
	Using this definition, we observe that the minimisation problem~\eqref{def.opt.decomp} admits a unique solution, so that $\{P_i v\}_{i=1}^m$ is well defined, as it corresponds to the minimisation of a strictly convex norm over a Hilbert space with a closed constraint.
	
	We next consider the map $A:\hat{H}\to H$ given by $A(v_i)_{i=1}^m=\sum\limits_{i=1}^mv_i$, which is clearly continuous. In particular, its kernel is closed. Furthermore, its range is $H$ by assumption. In this case, we may perform an orthogonal decomposition of $\hat{H}=\ker(A)\oplus\ker(A)^\perp$, where $A$ is injective on $\ker(A)^\perp$. In particular, we see that if $\tilde{A}:\text{ker}(A)^\perp\to H$ is the restriction of $A$, then $\tilde{A}$ is a continuous, linear bijection between Hilbert spaces, and therefore admits a continuous inverse $\tilde{A}^{-1}:H\to \ker(A)^{\perp}$. We note that $\tilde{A}^{-1}v$ solves the corresponding minimisation problem~\eqref{def.opt.decomp}, as we decompose $(v_i)_{i=1}^m=(v_i^\perp)_{i=1}^m \oplus (v_i^\parallel)_{i=1}^m \in \text{ker}(A)^\perp\oplus \text{ker}(A)$. Then $A(v_i)_{i=1}^m=A(v_i^\perp)_{i=1}^m$ and thus if $A(v_i)_{i=1}^m =v$, 
	\begin{equation}
		\|(v_i)_{i=1}^m \|^2_{\hat{H}}=\|(v_i^\perp)_{i=1}^m \|^2_{\hat{H}}+\|(v_i^\parallel)_{i=1}^m \|^2_{\hat{H}}\geq \|(v_i^\perp)_{i=1}^m\|^2_{\hat{H}}.
	\end{equation}
	We thus conclude that $\{P_i v\}_{i=1}^m=\tilde{A}^{-1}v$ is a well-defined bounded linear operator.  
	
	Finally, we observe that the operator norm of $\tilde{A}^{-1}$, which is necessarily finite, can be expressed as
	\begin{equation}\begin{split}
			\|\tilde{A}^{-1}\| =&\sup\limits_{v\in H}\frac{\|\tilde{A}^{-1}v\|_{\hat{H}}}{\|v\|_{H}}
			=\sup\limits_{v\in H}\frac{1}{\|v\|_{H}}\sqrt{\sum\limits_{i=1}^m \|P_iv\|_{H}^2}\\
			=&\sup\limits_{v\in H}\frac{1}{\|v\|_{H}}\min\left\{\sqrt{\sum\limits_{i=1}^m \|v_i\|_{H}^2}:v_i\in H,\, \sum\limits_{i=1}^m v_i=v\right\}=\xi \left( \{H_i\}_{i=1}^m \right)
		\end{split}
	\end{equation}
\end{proof}

\subsection{Proof of Proposition~\ref{prop:star}}

\begin{proof}
	For all $i=1,\dots m$, let $Q_i$ correspond to the orthogonal projection from $H$ onto $H_i$. Since $P_iH=H_i$ for all $i$ and $Q_i$ is the identity on $H_i$, then $Q_iP_i=P_i$. Furthermore, this implies that $P_i^T=P_i^TQ_i$, as $Q_i=Q_i^T$. 
	Consequently, one has that 
	\begin{equation}
		\| f \| _{H_i^*}= \| Q_if \| _{H^*} \quad \forall f\in H^*.
	\end{equation}
	
	To show the lower bound in~\eqref{eq:starbound}, it suffices to observe that since $Q_i$ is an orthogonal projection operator, it has an operator norm given by $1$, thus 
	\begin{equation}
		\| f \|_{\star}= \sqrt{\sum\limits_{i=1}^m \| f \|_{H_i^*}^2}= \sqrt{\sum\limits_{i=1}^m \| Q_if \| _{H^*}^2} \leq \sqrt{\sum\limits_{i=1}^m \| Q_i \|^2\| f \|_{H^*}^2}= \sqrt{m} \| f \|_{H^*}. 
	\end{equation}
	
	For the upper bound, we observe that $P_i^Tf = P_i^TQ_if $, for all $P_i \in \{P_i\}_{i=1}^m$. Thus, we obtain 
	\begin{equation}
		\begin{split}
			\| f \|_{H^*} =& \| \sum\limits_{i=1}^m P_i^Tf \|_{H^*}= \| \sum\limits_{i=1}^m P_i^TQ_if \|_{H^*} \leq \sum\limits_{i=1}^m \| P_i^TQ_if \|_{H^*} \\ 
			\leq & \sum\limits_{i=1}^m \| P_i \| \| Q_if \|_{H^*} \leq \left(\sum\limits_{i=1}^m \| P_i \| ^2\right)^\frac{1}{2}\left(\sum\limits_{i=1}^m \| Q_if \|_{H^*} ^2\right)^\frac{1}{2}.
		\end{split}
	\end{equation}
	Using Lemma~\ref{prop:maps}, and noting that $$ \left(\sum\limits_{i=1}^m \|P_i\|^2\right)^\frac{1}{2}=\xi \left( \{H_i\}_{i=1}^m \right),$$
	we conclude  $\| f \|_{H^*}  \leq \xi \left( \{H_i\}_{i=1}^m \right)\| f \|_{\star}$.

\end{proof}

\section{The partition of unity for the L-shape}\label{app:3}
In this appendix, we derive the constants of equivalence for the case of the L-shape domain guaranteed by Theorem~\ref{thm:Lshapepent_app}. We consider the cover of the domain as represented in Figure~\ref{figCase1Lshape} and denote $U_1=\Omega_1\setminus\overline{\Omega_2}=(-1,0)\times(0,1)$, $U_2=\Omega_2\setminus\overline{\Omega_1}=(0,1)\times(-1,0)$ and $U_{1,2}=\Omega_1\cap\Omega_2=U_{1,2}$, as shown in Figure~\ref{figCase1Lshape_app}.
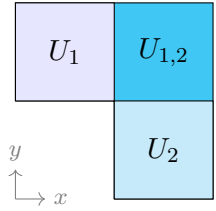
\begin{figure}[h!]
	\centering
	\begin{tikzpicture}[scale=1.3]
		\begin{scope}
			\draw[fill=cyan!60, opacity=1] (0,0) -- (-1,0) -- (-1,1) -- 
			(1,1) -- (1,-1) -- (0,-1) -- (0,0);
			\draw[fill=cyan!20,opacity=1] (0,-1)--(1,-1)--(1,0)--(0,0)--(0,-1);
			\draw[fill=blue!10,opacity=1] (-1,0)--(-1,1)--(0,1)--(0,0)--(-1,0);
			\node at (-0.5,0.5) {$U_1$};
			\node at (0.5,-0.5) {$U_2$};
			\node at (0.5,0.5) {$U_{1,2}$};
			{\footnotesize \draw[->,gray] (-1,-1) -- (-0.7,-1) node[right] {$x$};
				\draw[->,gray] (-1,-1) -- (-1,-0.7) node[above] {$y$};}
		\end{scope}
	\end{tikzpicture}
	\caption{The subdomains of the L-shape.}
	\label{figCase1Lshape_app}
\end{figure}

In the following analysis, we consider certain harmonic functions defined on $U_{1,2}$ that, for a given $v\in \Hoo$, are defined to be the weak solutions to:
\begin{equation}\label{eq:laplaciansv1v2}
	\begin{array}{r l l}
		\Delta v_1^h=&0 & \text{ in } U_{1,2}, \\
		v_1^h = & v & \text{ on } \Omega_1\cap \partial U_{1,2}, \\
		v_1^h = & 0 & \text{ on } \Omega_2\cap \partial U_{1,2}, \\
		v_1^h = & 0 & \text{ on } \partial\Omega\cap\partial U_{1,2},
	\end{array} \quad \text{ and } \quad  \begin{array}{r l l}
		\Delta v_2^h=&0 & \text{ in } U_{1,2},\\
		v_2^h = & 0 & \text{ on } \Omega_1\cap \partial U_{1,2},\\
		v_2^h = & v & \text{ on } \Omega_2\cap \partial U_{1,2},\\
		v_2^h = & 0 & \text{ on } \partial\Omega\cap\partial U_{1,2}.
	\end{array}	
\end{equation}

 We remark that as Theorem~\ref{thm:Lshapepent_app} guarantees that, for all $v\in \Hoo$, there exists $v_i\in H^1_0(\Omega_i)$ with $v_1+v_2=v$, the boundary conditions that appear in the PDEs in \eqref{eq:laplaciansv1v2} are all in $H^{\frac{1}{2}}$ of their respective domains, and thus the PDEs admit unique solutions. 

First, we have a preliminary result on the interior minimisation problem.
\begin{proposition}\label{prop:app3}
	Let $v\in \Hoo$ be given. Then the solutions $v_1,v_2$ to 
	\begin{equation}\label{eqMinLShapeConstant}\min\left\{\|v_1\|_{\Hoo}^2+\|v_2\|_{\Hoo}^2: v_1\in H^1_0(\Omega_1),\,v_2\in H^1_0(\Omega_2),\, v_1+v_2=v\right\}\end{equation}
	are given piecewise by 
	
	\begin{equation}
		v_1(\tx) = \left\{\begin{array}{l l}
			v & \tx\in U_1,\\
			\frac{1}{2}\left(v+v_1^h-v_2^h\right) & \tx\in U_{1,2},\\
			0 & \tx\in U_2,
		\end{array}\right. \text{ and } \,\, v_2(\tx)=\left\{\begin{array}{l l}
			0 & \tx\in U_1,\\
			\frac{1}{2}\left(v-v_1^h+v_2^h\right) & \tx\in U_{1,2},\\
			v & \tx\in U_2.
		\end{array}\right.
	\end{equation}
	
	Furthermore, the value of the minimum in \eqref{eqMinLShapeConstant} is given by 
	$$
	\int_\Omega|\nabla v|^2 d\tx+\frac{1}{2}\int_{U_{1,2}}|\nabla(v_1^h-v_2^h)|^2-|\nabla v|^2\dx.
	$$
\end{proposition}
\begin{proof}
	The conditions that $v_i\in H^1_0(\Omega_i)$ and $v_1+v_2=v$ immediately impose their values on $U_1$ and $U_2$, thus we need only obtain their behaviour on $U_{1,2}$. 
	We note that we may elimate $v_2$, as $v_2=v-v_1$, where $v_1=v$ on $U_1$ and $v_1=0$ on $U_2$. Thus we may reduce this to 
	\begin{equation}
		\min\left\{\|v_1\|_{\Hoo}^2+\|v-v_1\|_{\Hoo}^2: v_1\in H^1_0(\Omega_1),\,v_1|_{U_1}=v\right\},
	\end{equation}
	giving a minimisation problem over a single function. We have that 
	\begin{equation}
		\|v_1\|_{\Hoo}^2+\|v-v_1\|_{\Hoo}^2=\int_{\Omega}|\nabla v|^2\dx +\int_{U_{1,2}}2|\nabla v_1|^2-2\nabla v_1\cdot \nabla v\, d\tx,
	\end{equation}
	where we have the boundary condition that $v_1=v$ on $(\partial U_{1,2})\cap \Omega_1$ and $v_1=0$ on $\partial U_{1,2}\setminus \Omega_1$.
	
	The corresponding minimisation problem over $v_1$ therefore admits a unique solution $v_1$, which satisfies weakly the PDE
	\begin{equation}
		\begin{array}{r l l}
			\Delta v_1=&\frac{1}{2}\Delta v & \text{ in } U_{1,2},\\
			v_1 = & v & \text{ on } \Omega_1\cap \partial U_{1,2}, \\
			v_1 = & 0 & \text{ on } \Omega_2\cap \partial U_{1,2} \text{ and }  \partial\Omega\cap\partial U_{1,2}. 
		\end{array}
	\end{equation}
	Recalling \eqref{eq:laplaciansv1v2}, we may thus decompose $v_1$ as $v_1=\frac{1}{2}\left(v+v_1^h-v_2^h\right)$ on $U_{1,2}$. Correspondingly, $v_2=v-v_1 = \frac{1}{2}\left(v-v_1^h+v_2^h\right)$. Returning to the value of the integral,  
	\begin{equation}
		\begin{split}
			\int_{\Omega}|\nabla v|^2\dx +\int_{U_{1,2}}2|\nabla v_1|^2-2\nabla v_1\cdot \nabla v\, d\tx=\int_\Omega|\nabla v|^2+\frac{1}{2}\int_{U_{1,2}}|\nabla(v_1^h-v_2^h)|^2-|\nabla v|^2\dx.
		\end{split}
	\end{equation}
\end{proof}

We now turn to the analysis of the maximisation problem defining the operator norm. Consider $v$ to be fixed on $\partial U_{1,2}$, and separate the maximisation problem on each subdomain $U_1$, $U_2$, and $U_{1,2}$.

\begin{proposition}
	Let $V=\{v\in \Hoo:v|_{\Omega\setminus U_{1,2}}=v_0\}$. Then 
	\begin{equation}
		\max\limits_{v\in V} \left\lbrace  \frac{\|v_1\|_{H^1_0(\Om)}^2+\|v_2\|_{H^1_0(\Om)}^2}{\|v\|_{H^1_0(\Om)}^2} \right\rbrace  = 1+\frac{-2\int_{U_{1,2}}\nabla v_1^h\cdot\nabla v_2^h\dx}{\int_{U_{1,2}}|\nabla (v_1^h+v_2^h)|^2\dx + \int_{\Omega\setminus U_{1,2}}|\nabla v|^2\dx}.
	\end{equation}
\end{proposition}
\begin{proof}
	From our previous result, we have that 
	\begin{equation}\label{eqMaxProblem1}
		\begin{split}
			\max\limits_{v\in V} \left\lbrace \frac{\|v_1\|_{H^1_0(\Om)}^2+\|v_2\|_{H^1_0(\Om)}^2}{\|v\|_{H^1_0(\Om)}^2} \right\rbrace = & \max\limits_{v\in V} \left\lbrace  1+\frac{\int_{U_{1,2}}|\nabla (v^h_1-v^h_2)|^2-|\nabla v|^2\dx}{2\left(\int_{U_{1,2}}|\nabla v|^2\dx + \int_{\Omega\setminus U_{1,2}}|\nabla v|^2\dx\right)}\right\rbrace  .
		\end{split}
	\end{equation}
	As the integrals are non-negative, we have that the expression is a decreasing function of $\int_{U_{1,2}}|\nabla v|^2\dx$. In particular, the maximum is attained at the (unique) minimiser of $\min\limits_{v\in V}\left\lbrace \int_{U_{1,2}}|\nabla v|^2\dx\right\rbrace $. As $v$ is fixed on $\Omega\setminus U_{1,2}$, this prescribes a boundary condition, and we see that the minimiser of the Dirichlet energy --or, equivalently, the maximiser of \eqref{eqMaxProblem1}-- satisfies the PDE 
	\begin{equation}
		\begin{array}{r l l}
			\Delta v=&0 & \text{ in } U_{1,2},\\
			v = & v_0 & \text{ on } \partial U_{1,2}. 
		\end{array}
	\end{equation}
	Recalling \eqref{eq:laplaciansv1v2} and that $v_0=0$ on $\partial\Omega$, this means that $v = v_1^h+v_2^h$. Thus,
	\begin{equation}
		\max\limits_{v\in V} \left\lbrace 1+\frac{\int_{U_{1,2}}|\nabla (v^h_1-v^h_2)|^2-|\nabla v|^2\dx}{2\left(\int_{U_{1,2}}|\nabla v|^2\dx + \int_{\Omega\setminus U_{1,2}}|\nabla v|^2\dx\right)} \right\rbrace = 1+\frac{-2\int_{U_{1,2}}\nabla v_1^h\cdot\nabla v_2^h\dx}{\int_{U_{1,2}}|\nabla (v_1^h+v_2^h)|^2\dx + \int_{\Omega\setminus U_{1,2}}|\nabla v|^2\dx}.
	\end{equation}
\end{proof}


\begin{proposition}
	Let $V_2=\{v\in \Hoo:v|_{\partial U_{1,2}}=v_0\}$. Then
	\begin{equation}
		\begin{split}
			\sup\limits_{v\in V_2} \left\lbrace  1+\frac{-2\int_{U_{1,2}}\nabla v_1^h\cdot\nabla v_2^h\dx}{\int_{U_{1,2}}|\nabla (v_1^h+v_2^h)|^2\dx + \int_{\Omega\setminus U_{1,2}}|\nabla v|^2\dx} \right\rbrace  = \\ \max\left(1,\frac{\int_{U_{1,2}}|\nabla v_1^h|^2+|\nabla v_2^h|^2\dx}{\int_{U_{1,2}}|\nabla v_1^h|^2+|\nabla v_2^h|^2+\nabla v_1^h\cdot\nabla v_2^h\dx}\right)
		\end{split}
	\end{equation}
\end{proposition}
\begin{proof}
	The analysis of the supremum depends on the sign of $\int_{U_{1,2}}\nabla v_1^h\cdot v_2^h\dx$. As the term appearing in the denominator is non-negative, if $\int_{U_{1,2}}\nabla v_1^h\cdot v_2^h\dx\geq 0$, then the expression is bounded by $1$, and by taking $\int_{\Omega\setminus U_{1,2}}|\nabla v|^2\dx$ large, we see that the supremum of 1 is attainable. Assuming $\int_{U_{1,2}}\nabla v_1^h\cdot v_2^h\dx<0$, then we have that the expression is a decreasing function of $\int_{\Omega\setminus U_{1,2}}|\nabla v|^2\dx$, thus the maximiser corresponds to the minimiser of the seminorm. The restriction that $v\in V_2$ prescribes boundary conditions, and by a similar argument as before we have that on $U_1$, $v$ must satisfy the PDE 
	\begin{equation}
		\begin{array}{r l l}
			\Delta v=&0 & \text{ in } U_{1}\\
			v= & v_0 & \text{ on } \Omega_1\cap\partial U_1\\
			v = & 0 & \text{ on } \Omega_2\cap\partial U_1 \text{ and } \partial\Omega\cap\partial U_{1}. 
		\end{array}
	\end{equation}
	weakly. Due to symmetry considerations, as $U_1=\{(-x,y):(x,y)\in U_{1,2}\}$, with the correspodning mapping preserving boundary conditions, we then have that $v(x,y)=v_1^h(-x,y)$ on $U_1$. Similarly, we have that $v(x,y)=v_2^h(x,-y)$ on $U_2$. In particular, $\int_{\Omega\setminus U_{1,2}}|\nabla v|^2\dx = \int_{U_{1,2}}|\nabla v_1^h|^2+|\nabla v_2^h|^2\dx$. Thus 
	{\footnotesize \begin{equation}
		\begin{split}
			\sup \limits_{v\in V_2} \left\lbrace1+\frac{-2\int_{U_{1,2}}\nabla v_1^h\cdot\nabla v_2^h\dx}{\int_{U_{1,2}}|\nabla (v_1^h+v_2^h)|^2\dx + \int_{\Omega\setminus U_{1,2}}|\nabla v|^2\dx} \right\rbrace &= \max\left(1,1+\frac{-2\int_{U_{1,2}}\nabla v_1^h\cdot\nabla v_2^h\dx}{\int_{U_{1,2}}|\nabla (v_1^h+v_2^h)|^2+|\nabla v_1^h|^2+|\nabla v_2^h|^2\dx}\right)\\
			= \max\left(1,1+\frac{-2\int_{U_{1,2}}\nabla v_1^h\cdot\nabla v_2^h\dx}{2\int_{U_{1,2}}|\nabla v_1^h|^2+|\nabla v_2^h|^2+\nabla v_1^h\cdot\nabla v_2^h\dx}\right) &=  \max\left(1,\frac{\int_{U_{1,2}}|\nabla v_1^h|^2+|\nabla v_2^h|^2\dx}{\int_{U_{1,2}}|\nabla v_1^h|^2+|\nabla v_2^h|^2+\nabla v_1^h\cdot\nabla v_2^h\dx}\right)
		\end{split}
	\end{equation}}
\end{proof}

Then, our main theorem is the following: 
\begin{theorem}\label{thmConstantL}
	For the L-shape domain, $\xi \left( \{H_i\}_{i=1}^2\right)\leq \sqrt{2}$.
\end{theorem}
\begin{proof}
	Due to the preceding propositions, we have that 
	\begin{equation}
	 \xi \left( \{H_i\}_{i=1}^2\right)^2=\max\limits_{v_0 \in H^\frac{1}{2}(U_{1,2})}\max\left(1,\frac{\int_{U_{1,2}}|\nabla v_1^h|^2+|\nabla v_2^h|^2\dx}{\int_{U_{1,2}}|\nabla v_1^h|^2+|\nabla v_2^h|^2+\nabla v_1^h\cdot\nabla v_2^h\dx}\right)
	\end{equation}
	we then have that 
	\begin{equation}
		\begin{split}
			\frac{\int_{U_{1,2}}|\nabla v_1^h|^2+|\nabla v_2^h|^2\dx}{\int_{U_{1,2}}|\nabla v_1^h|^2+|\nabla v_2^h|^2+\nabla v_1^h\cdot\nabla v_2^h\dx}=&\left(1+\frac{\int_{U_{1,2}}\nabla v_1^h\cdot\nabla v_2^h\dx}{\int_{U_{1,2}}|\nabla v_1^h|^2+|\nabla v_2^h|^2\dx}\right)^{-1}
		\end{split}
	\end{equation}
	By a combination of Cauchy-Schwartz and Young's inequality we see, 
	\begin{equation}\label{eqCSYoung}
		|\langle v_1^h,v_2^h\rangle_{H^1}|\leq  \|v_1^h\|_{H^1}\|v_2^h\|_{H^1}\leq \frac{1}{2}\left(\|v_1^h\|_{H^1}^2+\|v_2^h\|_{H^1}^2\right), 
	\end{equation}
	so that $\frac{\int_{U_{1,2}}\nabla v_1^h\cdot\nabla v_2^h\dx}{\int_{U_{1,2}}|\nabla v_1^h|^2+|\nabla v_2^h|^2\dx} \in \left(-\frac{1}{2},\frac{1}{2}\right)$. Then, as $x\mapsto (1+x)^{-1}$ is a decreasing function on $\left(-\frac{1}{2},\frac{1}{2}\right)$, we have that 
	\begin{equation}
		\left(1+\frac{\int_{U_{1,2}}\nabla v_1^h\cdot\nabla v_2^h\dx}{\int_{U_{1,2}}|\nabla v_1^h|^2+|\nabla v_2^h|^2\dx}\right)^{-1}\leq \left(1-\frac{1}{2}\right)^{-1}=2. 
	\end{equation}
	Thus $ \xi \left( \{H_i\}_{i=1}^2\right)^2\leq \max(1,2)=2$, giving the result. 
\end{proof}

We note via the proof of Theorem~\ref{thmConstantL}, that the constant $\sqrt{2}$ may not indeed be sharp. The inequalities in \eqref{eqCSYoung} are equalities if and only if $v_1^h=\pm v_2^h$. This, however, is incompatible with the boundary conditions prescribed on $v_1^h$ and $v_2^h$.

 \end{document}